\documentclass[11pt,reqno]{amsart}

\usepackage[T1]{fontenc}
\usepackage{lmodern}
\usepackage{microtype}
\usepackage{amsmath,amssymb,amsthm,mathtools}
\usepackage{aliascnt}
\usepackage{bm}
\usepackage{enumitem}
\usepackage{tikz}
\usepackage{float}
\usetikzlibrary{arrows.meta,calc,patterns,positioning}
\usepackage[hidelinks]{hyperref}
\usepackage[nameinlink,capitalise,noabbrev]{cleveref}
\usepackage{cite}

\allowdisplaybreaks
\numberwithin{equation}{section}

\newtheorem{theorem}{Theorem}[section]
\newaliascnt{proposition}{theorem}
\newtheorem{proposition}[proposition]{Proposition}
\aliascntresetthe{proposition}
\newaliascnt{lemma}{theorem}
\newtheorem{lemma}[lemma]{Lemma}
\aliascntresetthe{lemma}
\newaliascnt{corollary}{theorem}
\newtheorem{corollary}[corollary]{Corollary}
\aliascntresetthe{corollary}
\newaliascnt{conjecture}{theorem}

\aliascntresetthe{conjecture}

\theoremstyle{definition}
\newaliascnt{definition}{theorem}
\newtheorem{definition}[definition]{Definition}
\aliascntresetthe{definition}
\newaliascnt{example}{theorem}
\newtheorem{example}[example]{Example}
\aliascntresetthe{example}

\theoremstyle{remark}
\newaliascnt{remark}{theorem}
\newtheorem{remark}[remark]{Remark}
\aliascntresetthe{remark}

\crefname{theorem}{Theorem}{Theorems}
\crefname{proposition}{Proposition}{Propositions}
\crefname{lemma}{Lemma}{Lemmas}
\crefname{corollary}{Corollary}{Corollaries}
\crefname{conjecture}{Conjecture}{Conjectures}
\crefname{definition}{Definition}{Definitions}
\crefname{example}{Example}{Examples}
\crefname{remark}{Remark}{Remarks}

\newcommand{\N}{\mathbb N}
\newcommand{\Nzero}{\mathbb N_0}
\newcommand{\Z}{\mathbb Z}
\newcommand{\R}{\mathbb R}
\newcommand{\C}{\mathbb C}

\newcommand{\Torus}{\mathbb T}
\newcommand{\Log}{\operatorname{Log}}
\newcommand{\Hh}{\mathcal H}
\newcommand{\K}{\mathcal K}
\newcommand{\B}{\mathcal B}

\newcommand{\M}{\mathcal M}
\newcommand{\cN}{\mathcal N}

\newcommand{\Ran}{\operatorname{Ran}}
\newcommand{\Ker}{\operatorname{Ker}}
\newcommand{\closspan}{\overline{\operatorname{span}}}
\newcommand{\supp}{\operatorname{supp}}
\newcommand{\rank}{\operatorname{rank}}

\newcommand{\dd}{\,\mathrm d}
\newcommand{\ip}[2]{\left\langle #1,#2\right\rangle}
\newcommand{\norm}[1]{\left\lVert #1\right\rVert}

\title[Sectorial dynamical frames]
{Sectorial Finitely Generated \\ Dynamical Frames}

\author{Ilya Krishtal}
\address{School of Mathematical and Statistical Sciences, Northern Illinois University, DeKalb, IL 60115, USA}
\email{ikrishtal@niu.edu}
\thanks{Corresponding author: Ilya Krishtal.}

\author{Javad Mashreghi}
\address{D\'epartement de math\'ematiques et de statistique, Universit\'e Laval, Qu\'ebec (Qu\'ebec) G1V 0A6, Canada}
\email{javad.mashreghi@mat.ulaval.ca}

\author{Brendan Miller}
\address{Department of Mathematics, Vanderbilt University, 1326 Stevenson Center, Nashville, TN 37240, USA}
\email{brendan.miller@vanderbilt.edu}

\date{September 6, 2026}


\hypersetup{
  pdftitle={Sectorial Finitely Generated Dynamical Frames},
  pdfauthor={Ilya Krishtal, Javad Mashreghi, Brendan Miller},
  pdfsubject={Compact Fredholm localization, sectorial temporal sampling, and confluent functional calculus for finitely generated dynamical frames},
  pdfkeywords={dynamical frame, compact localization, Calkin algebra, Fredholm index, sectorial spectrum, nonharmonic Fourier series, generalized divided differences, phase retrieval, functional calculus}
}

\begin{document}

\begin{abstract}
We study frames of unilateral iterations generated by finitely many vectors
under an invertible bounded operator $T$.  Suppose that the spectrum of $T$
lies in $\{\rho e^{it}:0<\rho\le1,\ |t|\le c\}$, where $c<\pi$, and that
its essential spectrum lies in $\{e^{it}:|t|\le c_{\mathrm e}\}$, where
$0\le c_{\mathrm e}\le c$.  Real powers are defined through the principal
logarithm.  Our structural result identifies fractional-orbit synthesis,
up to an invertible map and a compact perturbation, with synthesis of
projected vector-valued exponentials on an arc containing the essential
spectrum.  Semi-Fredholm properties and Fredholm index therefore transfer.

For a uniformly separated temporal set $\Lambda$, this reduction gives a
frame when its one-sided lower uniform density satisfies
$D_+^-(\Lambda)>c_{\mathrm e}/\pi$ and its logarithmic block density satisfies
$L(\Lambda)>c/\pi$.  Completeness at the latter threshold requires only
an invertible operator with spectrum in the corresponding angular sector
and a complete integer multiorbit.
At $L(\Lambda)>2c/\pi$, the complement property is equivalent to a finite
condition in channel-cyclic subspaces, with real phase-retrieval
consequences.

Generalized divided differences admit the same compact reduction and
the analogous frame theorem with multiplicities counted.  For arbitrary
one-point-per-cell sampling with cell length $r$, the sufficient raw-sampling
conditions are $rc<\pi$ and $2rc_{\mathrm e}<\pi$. Thus, the full range
$rc<\pi$ holds for raw samples when the essential spectrum is contained
in $\{1\}$; over the unrestricted class, the sharp universal range is
$rc<\pi/2$.  Divided-difference preconditioning restores $rc<\pi$ for
all operators considered and converts exact collisions into
logarithmic-derivative data.
\end{abstract}

\maketitle

\section{Introduction}

Let $T$ be a bounded operator on a complex Hilbert space $\Hh$, and let
$f_1,\ldots,f_m\in\Hh$.  A \emph{frame of iterations}, or dynamical frame,
is a family
\begin{equation}\label{eq:intro-frame}
  \mathcal F(T,\bm f)
  =\{T^nf_j:n\in\mathbb N_0,\ 1\le j\le m\}
\end{equation}
for which there exist $0<a\le b<\infty$ such that
\begin{equation}\label{eq:intro-frame-def}
 a\|x\|^2\le
 \sum_{n=0}^\infty\sum_{j=1}^m|\langle x,T^nf_j\rangle|^2
 \le b\|x\|^2,
 \qquad x\in\mathcal H.
\end{equation}
We call the index $j$ a \emph{channel} and refer to
\eqref{eq:intro-frame} as the integer multiorbit.
Frames of this form arise in dynamical sampling and evolving-data problems;
see \cite{ACMT17,ACKM26}.  Their finite-generator structure is encoded by
vector-valued Hardy-space models \cite{CMS23,ACNP26}.  The purpose of this
paper is to identify an operator-theoretic mechanism that transfers temporal
sampling questions for general, possibly nonnormal, dynamical frames to
nonharmonic Fourier analysis.

Our main contribution is a compact Fredholm localization theorem.  Given the
basic Hardy model $(A,g_1,\ldots,g_m)$ of \eqref{eq:intro-frame}, a temporal
family $\bm s=(s_k)$ with bounded multiplicity of its integer parts, and an arc
$(-d,d)$ containing the essential spectral angles, we construct a closed
localized Hardy range $\mathcal Y_d$ and an invertible operator
$V_d:\mathcal Y_d\to\mathcal H$ such that
\begin{equation}\label{eq:intro-compact-equivalence}
  D_{T,\bm s}=V_dD_{\mathcal E_d(\bm s)}+K_{\bm s},
  \qquad K_{\bm s}\ \text{compact}.
\end{equation}
Here, $D_{T,\bm s}$ is the fractional-orbit synthesis operator and
$D_{\mathcal E_d(\bm s)}$ is the synthesis operator of the projected
vector-valued exponential family in $\mathcal Y_d$.  In particular, upper and
lower semi-Fredholm properties pass between the two systems, and their
Fredholm indices agree.  This compact equivalence, proved in
\cref{thm:fredholm-transfer-general}, is the functional-analytic core of the
paper.  The density and perturbation theorems combine this compact
equivalence with classical sampling and uniqueness results.

Throughout the paper
$\mathbb N_0=\{0,1,2,\ldots\}$, $\mathbb T$ is the unit circle, and $P_Y$ denotes the
orthogonal projection onto a closed subspace $Y$.
Circle and arc $L^2$ spaces use normalized angular measure $dt/(2\pi)$.
For the frame results we assume that $T$ is invertible and, after
multiplication by a unimodular constant, that
\begin{equation}\label{eq:intro-sector}
  \sigma(T)\subset\Sigma_c
  :=\{\rho e^{it}:0<\rho\le1,\ |t|\le c\},
  \qquad 0\le c<\pi.
\end{equation}
The principal logarithm $\Log T$ then defines
$T^z=e^{z\Log T}$ for $z\in\mathbb C$.  We distinguish the full spectral bound
$c$ from an essential spectral bound $c_{\mathrm e}$ satisfying
\begin{equation}\label{eq:essential-sector-general}
  \sigma_{\mathrm e}(T)\subset\{e^{it}:|t|\le c_{\mathrm e}\},
  \qquad 0\le c_{\mathrm e}\le c.
\end{equation}
Here $\sigma_{\mathrm e}$ is the spectrum modulo compact operators,
defined in \cref{sec:model}.  The Hardy model shows that
$c_{\mathrm e}=c$ is always admissible.  If the essential spectrum is
empty, as in finite dimension, we may take $c_{\mathrm e}=0$.
For the completeness and complement-property results, we allow the
larger angular sector
\begin{equation}\label{eq:angular-sector-general}
  \widehat\Sigma_c:=\{\rho e^{it}:\rho>0,\ |t|\le c\},
  \qquad 0\le c<\pi,
\end{equation}
without a frame or power-boundedness assumption.

For a locally finite
$\Lambda\subset[0,\infty)$, its one-sided lower uniform density is
\begin{equation}\label{eq:intro-density-definition}
 D_+^-(\Lambda)
 =\liminf_{R\to\infty}\ \inf_{x\ge0}
   \frac{\#(\Lambda\cap[x,x+R])}{R}.
\end{equation}
We call $\Lambda$ uniformly separated when
$\inf_{\lambda\ne\mu}|\lambda-\mu|>0$.

Our first group of applications establishes completeness under weaker conditions on $T$.  If $T$ is invertible with $\sigma(T)\subset\widehat\Sigma_c$,
the integer multiorbit is complete, and each channel-dependent temporal set
satisfies $L(\Lambda_j)>c/\pi$, then the sampled multiorbit is complete.
Here $L$ is the logarithmic block density defined in
\eqref{eq:log-block-density-general}; separation is unnecessary. We also study the complement property of sampled multiorbits.  For a common set with $L(\Lambda)>2c/\pi$, it is shown to be equivalent to a finite complement condition on the channel-cyclic subspaces.

Under the frame hypotheses, separated sampling has closed synthesis range
of finite codimension whenever $D_+^-(\Lambda)>c_{\mathrm e}/\pi$.
The additional inequality $L(\Lambda)>c/\pi$ gives completeness and
therefore a frame.  In particular, $D_+^-(\Lambda)>c/\pi$ suffices, and
every progression $r\mathbb N_0$ is admissible when $rc<\pi$.
The examples also establish sharpness of the universal completeness and
complement-property constants.

The second group of applications concerns perturbations and collisions.
Bounded perturbations are handled by residue-class extraction;  
\cref{rem:kadec-avdonin-general}
places this result relative to the classical Kadec and Avdonin criteria.  When
temporal nodes form uniformly bounded clusters, generalized divided differences
(GDD) provide the correct confluent data. At an $M$-fold
collision, the data become
\[
  T^\lambda f_j,\quad T^\lambda(\Log T)f_j,\quad\ldots,\quad
  \frac{1}{(M-1)!}T^\lambda(\Log T)^{M-1}f_j.
\] 
We prove a
GDD counterpart of \eqref{eq:intro-compact-equivalence} and obtain a frame
under the same pair of density conditions, with multiplicities counted.
In particular, multiplicity-counted lower uniform density above $c/\pi$
suffices.  
Consequently, arbitrary one-point-per-cell sets are stable after GDD
preconditioning throughout $rc<\pi$.  For raw samples the
conditions are $rc<\pi$ and $2rc_{\mathrm e}<\pi$.  They recover the
universal range $rc<\pi/2$ and give the full range $rc<\pi$ when
$\sigma_{\mathrm e}(T)\subseteq\{1\}$. 

Recent work on operator-orbit frames includes optimal Parseval realizations and
the frame index \cite{ACNP26}, hyperinvariant-subspace methods
\cite{BHKLL25}, non-surjective orbit frames \cite{BW24}, moment systems for
positive exponent sets \cite{GGP26}, Hardy-space multiplication models
\cite{AC26}, and normal-operator constructions \cite{KP26}.  Functional
calculus and perturbation questions appear in \cite{ElIdrissi26,RV26}, while
spectral-preserver and positive-power stability results are developed in
\cite{WuBR26,WuPower26}.  The closest antecedents are the block diagonal
Carleson (BDC) theory \cite{KMBDC26} and the diagonal functional-calculus theory
\cite{KMMFFC26}.  The BDC theory uses interpolation and Jordan-block structure to allow
global natural-density conditions.  Our essential-spectrum refinement
recovers the one-point-per-cell conclusion of \cite[Theorem~1.3]{KMMFFC26}
under the broader assumption $\sigma_{\mathrm e}(T)\subseteq\{1\}$.
The operator need not be normal or diagonalizable, and the model may have
arbitrary finite multiplicity.  For general essential spectra, GDD data
provide an additional way to handle temporal collisions.

The paper is organized as follows.  \Cref{sec:model} develops the Hardy model,
continuous-symbol Calkin calculus, and the compact Fredholm localization
principle.  \Cref{sec:separated} proves the completeness, complement-property, and
separated sampling theorems.  \Cref{sec:perturbations} treats bounded temporal
perturbations and records the Kadec--Avdonin comparison.  \Cref{sec:gdd}
develops confluent localization and GDD sampling.  \Cref{sec:sharpness}
contains the sharpness and exact-collision examples.

\section{The Hardy model and compact Fredholm localization}
\label{sec:model}

\subsection{The basic vector-valued model}

Let $H^2(\C^m)$ denote the Hardy space of analytic
$\C^m$-valued functions
\[
   h(z)=\sum_{n=0}^\infty h_nz^n,
   \qquad \sum_{n=0}^\infty\norm{h_n}_{\C^m}^2<\infty.
\]
Let $S_m$ be multiplication by $z$, and let $e_1,\ldots,e_m$ be the standard
basis of $\C^m$, regarded as constant functions in $H^2(\C^m)$.  We identify
$H^2(\C^m)$ with $\ell^2(\Nzero\times\{1,\ldots,m\})\simeq \ell^2(\Nzero)\otimes\C^m$ through the orthonormal
basis $(z^ne_j)_{n,j}$.

Recall that a family $(x_i)_{i\in I}$ is a frame for $\Hh$ if there are
$0<a\le b<\infty$ such that
\[
   a\norm{x}^2\le\sum_{i\in I}|\ip{x}{x_i}|^2\le b\norm{x}^2,
   \qquad x\in\Hh.
\]
Equivalently, its synthesis operator is bounded and onto \cite{Christensen16}. A frame is called Parseval if the frame bounds satisfy $a=b=1$.

Assume throughout this section that \eqref{eq:intro-frame} is a frame.  Its
synthesis operator, under the aforementioned identification  $H^2(\C^m)\simeq\ell^2(\Nzero\times\{1,\ldots,m\})$, is
\begin{equation}\label{eq:synthesis-general}
   D_{\bm f}:H^2(\C^m)\longrightarrow\Hh,
   \qquad D_{\bm f}(z^ne_j)=T^nf_j.
\end{equation}
It satisfies the intertwining identity
\begin{equation}\label{eq:module-intertwining}
   D_{\bm f}S_m=TD_{\bm f}.
\end{equation}
Set
\begin{equation}\label{eq:model-subspaces}
   \M=\Ker D_{\bm f},
   \qquad \cN=H^2(\C^m)\ominus\M,
\end{equation}
and write $P=P_{\cN}$.

The following proposition describes a basic model of a finitely generated dynamical frame.

\begin{proposition}
\label{prop:basic-model}
The subspace $\M$ is $S_m$-invariant and $\cN$ is $S_m^*$-invariant.  Put
\begin{equation}\label{eq:A-basic}
   A=PS_m|_{\cN},
   \qquad g_j=Pe_j,
   \qquad X=D_{\bm f}|_{\cN}.
\end{equation}
Then $X:\cN\to\Hh$ is boundedly invertible,
\begin{equation}\label{eq:basic-similarity}
   XA=TX,
   \qquad Xg_j=f_j,
\end{equation}
and
\begin{equation}\label{eq:basic-parseval}
   \{A^ng_j:n\in\Nzero,\ 1\le j\le m\}
\end{equation}
is a Parseval frame for $\cN$.  In fact,
\begin{equation}\label{eq:basic-columns}
   A^ng_j=P(z^ne_j).
\end{equation}
\end{proposition}

\begin{proof}
The invariance of $\M$ follows from \eqref{eq:module-intertwining}.  Since
$D_{\bm f}$ is onto, its restriction to $\M^\perp$ is a bounded bijection onto
$\Hh$ and therefore has a bounded inverse.  For $h\in\cN$,
\[
   XAh=D_{\bm f}PS_mh=D_{\bm f}S_mh=TD_{\bm f}h=TXh,
\]
because $D_{\bm f}$ vanishes on $\M$.  The identity $Xg_j=f_j$ is immediate.
Since $\cN$ is $S_m^*$-invariant, powers of the compression satisfy
$A^nP=PS_m^nP$; using the invariance of $\M$ gives
$PS_m^nP=PS_m^n$ and hence \eqref{eq:basic-columns}.  Finally,
\[
 \sum_{n=0}^\infty\sum_{j=1}^m
   |\ip h{A^ng_j}|^2
 =\sum_{n,j}|\ip h{P(z^ne_j)}|^2
 =\sum_{n,j}|\ip h{z^ne_j}|^2
 =\norm h^2,
\]
for every $h\in\cN$.
\end{proof}

This model is the finite-multiplicity version of the scalar compressed-shift
representation; see, for example, \cite{Nikolski02,Sarason94} for scalar model
spaces and \cite{CMS23} for the vector-valued frame model.  We shall prove all temporal sampling statements for the basic
Parseval frame and transfer them through $X$.

We write $\mathcal B(X,Y)$ and $\mathcal K(X,Y)$ for the bounded and
compact operators between Hilbert spaces, omitting $Y$ when $X=Y$.
A bounded operator $B:X\to Y$ is upper semi-Fredholm if its range is
closed and $\dim\Ker B<\infty$, and lower semi-Fredholm if its range is
closed and $\dim(Y/\Ran B)<\infty$.  It is Fredholm if both conditions
hold; its index is
$\operatorname{ind}B=\dim\Ker B-\dim(Y/\Ran B)$.

For $R\in\B(\cN)$, we write $\dot R$ for its coset in the Calkin algebra
$\B(\cN)/\K(\cN)$ and set
$\sigma_{\mathrm e}(R)=\sigma(\dot R)$.  From now on, assume that $T$ is
invertible and satisfies
\eqref{eq:intro-sector}.  The model operator $A$ is similar to $T$, and similarity preserves
both the spectrum and the essential spectrum.  Thus
\begin{equation}\label{eq:A-sector}
   \sigma(A)=\sigma(T)\subset\Sigma_c.
\end{equation}
Let
\[
   L_A=\Log A,
   \qquad A^\alpha=e^{\alpha L_A},\quad \alpha\in\C.
\]
Holomorphic functional calculus and \eqref{eq:basic-similarity} give
\begin{equation}\label{eq:fractional-similarity}
   T^sX=XA^s,
   \qquad s\in\R.
\end{equation}

\subsection{Finite defects and essential unitarity}

The following finite defect lemma is a key to the subsequent analysis. Recall that $\sigma (A) \subseteq \Sigma_c$ for some $c\in [0,\pi)$ as in \eqref{eq:A-sector}.

\begin{lemma}
\label{lem:finite-defect}
The two defect operators of $A$ have finite rank:
\begin{equation}\label{eq:defect-one}
   I-AA^*=P P_{\C^m}|_{\cN},
   \qquad \rank(I-AA^*)\le m,
\end{equation}
and
\begin{equation}\label{eq:defect-two}
   I-A^*A=B^*B,
   \qquad B=P_{\M}S_m|_{\cN},
   \qquad \rank(I-A^*A)\le m.
\end{equation}
Consequently, the Calkin image $\dot A$ is unitary, and  its spectrum satisfies
\begin{equation}\label{eq:essential-spectrum-E}
   E_A: =\sigma(\dot A)=\sigma_{\mathrm e}(A) \subseteq   \{e^{it}:|t|\le c\}.
\end{equation}
\end{lemma}

\begin{proof}
Since $\cN$ is $S_m^*$-invariant by \cref{prop:basic-model}, we have $A^*=S_m^*|_{\cN}$.  As
$S_mS_m^*=I-P_{\C^m}$,
\[
   I-AA^*=P(I-S_mS_m^*)|_{\cN}=P P_{\C^m}|_{\cN},
\]
which has rank at most $m$.  Also,
\[
   I-A^*A
   =S_m^*(I-P)S_m|_{\cN}
   =(P_{\M}S_m|_{\cN})^*(P_{\M}S_m|_{\cN}).
\]
For $x\in\cN$ and $y\in\M$,
\[
   \ip{P_{\M}S_mx}{S_my}
   =\ip{S_mx}{S_my}=\ip{x}{y}=0.
\]
Thus $\Ran B\subset\M\ominus S_m\M$.  By the Beurling--Lax theorem, the wandering
subspace $\M\ominus S_m\M$ has dimension at most $m$.  This proves the rank
bound.  Hence $\dot A$ is both an isometry and a coisometry in the Calkin
algebra.  Its spectrum lies on the unit circle and, by
$E_A\subset\sigma(A)$ together with \eqref{eq:A-sector}, lies in the stated arc.
\end{proof}

\subsection{Continuous symbols}

Let $P_+$ be the Riesz projection on $L^2(\Torus;\C^m)$, and let
$M_q$ denote multiplication by a scalar function $q$ on that space.  For
$q\in C(\Torus)$ define
\begin{equation}\label{eq:Toeplitz-compressions}
   T_q=P_+M_q|_{H^2(\C^m)},
   \qquad A_q=PT_q|_{\cN}.
\end{equation}
The symbol is scalar and acts identically in every channel.

The following lemma lets us define a continuous-symbol Calkin calculus.

\begin{lemma}
\label{lem:continuous-symbol-calculus}
For every $q\in C(\Torus)$,
\begin{equation}\label{eq:Calkin-symbol}
   \dot A_q=q(\dot A).
\end{equation}
Hence $A_q$ is compact whenever $q|_{E_A}=0$.  Moreover,
\begin{equation}\label{eq:intertwining-symbol-compact}
   A_qP-PT_q:H^2(\C^m)\longrightarrow\cN
\end{equation}
is compact.
\end{lemma}

\begin{proof}
For $n\ge0$, invariance of $\M$ gives
\(
   A_{z^n}=A^n.
\)
Since $\cN$ is $S_m^*$-invariant,
\(
   A_{\overline z^{\,n}}=A^{*n}.
\)
Thus \eqref{eq:Calkin-symbol} holds for trigonometric polynomials and then for
all continuous $q$ by uniform approximation.  The compactness conclusion follows
from continuous functional calculus for the unitary $\dot A$.

For \eqref{eq:intertwining-symbol-compact}, the difference is zero for
$q=z^n$.  For $q=\overline z$, note that
\begin{equation}\label{eq:commutator-backward}
   A_{\overline z}P-PT_{\overline z}
   =S_m^*P-PS_m^*=-PS_m^*P_{\M}.
\end{equation}
The map on the right vanishes on $S_m\M$, so it factors through the wandering
space $\M\ominus S_m\M$ and has rank at most $m$.  A telescoping commutator
identity gives finite rank for $q=\overline z^{\,n}$.  Uniform approximation
again yields compactness for every continuous symbol.
\end{proof}

Fix an essential spectral bound $c_{\mathrm e}$ as in
\eqref{eq:essential-sector-general}; thus
$E_A=\sigma_{\mathrm e}(T)\subset\{e^{it}:|t|\le c_{\mathrm e}\}$.
Given any $d$ with
\begin{equation}\label{eq:cutoff-choice-general}
   c_{\mathrm e}<d<\pi,
\end{equation}
choose a real-valued $\chi=\chi_d\in C^\infty(\Torus)$ satisfying, in angular
coordinates $t\in(-\pi,\pi)$,
\begin{equation}\label{eq:cutoff-general}
   \chi(e^{it})=1\text{ on a neighborhood of }[-c_{\mathrm e},c_{\mathrm e}],
   \text{ and } \supp\chi\subset\{e^{it}:|t|<d\}.
\end{equation}
For $\alpha\in\C$, set
\begin{equation}\label{eq:q-alpha-general}
   q_\alpha(e^{it})=\chi(e^{it})e^{i\alpha t}.
\end{equation}

In the following lemma, we obtain a compact localization of fractional powers.

\begin{lemma}
\label{lem:compact-localization-general}
For every $\alpha\in\C$, the operator
\begin{equation}\label{eq:C-alpha-general}
   C_\alpha=A^\alpha P-PT_{q_\alpha}:
   H^2(\C^m)\longrightarrow\cN
\end{equation}
is compact, and $\alpha\mapsto C_\alpha$ is an entire compact-operator-valued
function.

Given a countable index set $I$, let $(s_k)_{k\in I}\subset[0,\infty)$ satisfy
\[
   s_k=n_k+\alpha_k,
   \qquad n_k\in\Nzero,
   \qquad 0\le\alpha_k<1,
\]
and suppose that
\begin{equation}\label{eq:integer-multiplicity-general}
   \sup_{n\in\Z}\#\{k:n_k=n\}<\infty.
\end{equation}
Define, for $1\le j\le m$,
\begin{equation}\label{eq:localized-columns-general}
   x_{k,j}=A^{s_k}g_j,
   \qquad
   w_{k,j}=PT_{q_{\alpha_k}}(z^{n_k}e_j).
\end{equation}
The synthesis operators of $(x_{k,j})$ and $(w_{k,j})$ differ by a compact
operator.  For $h\in\cN$,
\begin{equation}\label{eq:localized-coefficient-general}
   \ip h{w_{k,j}}
   =\int_{-d}^d
     \chi(e^{it})\,h_j(e^{it})e^{-is_kt}\,\frac{\dd t}{2\pi}.
\end{equation}
Finitely many exceptional real times change the synthesis comparison only by a
finite-rank operator.
\end{lemma}

\begin{proof}
On the essential spectrum $E_A$, the principal power $z^\alpha$ is continuous and
$q_\alpha(z)=z^\alpha$.  By \cref{lem:continuous-symbol-calculus},
\[
   A^\alpha-A_{q_\alpha}\in\K(\cN),
\]
where the power is defined by holomorphic functional calculus.  Combining
this with \eqref{eq:intertwining-symbol-compact} proves compactness of
$C_\alpha$.  The map $\alpha\mapsto A^\alpha P$ is entire because
$A^\alpha=e^{\alpha\Log A}$.  The map
$\alpha\mapsto q_\alpha$ is entire with values in $C(\Torus)$: its
$\ell$th derivative is $(it)^\ell\chi(e^{it})e^{i\alpha t}$ in angular
coordinates.  Hence $\alpha\mapsto PT_{q_\alpha}$, and therefore
$\alpha\mapsto C_\alpha$, is entire in operator norm.  Since every
$C_\alpha$ is compact and the compact operators form a norm-closed subspace,
this is an entire $\K(H^2(\C^m),\cN)$-valued map.

Consider the Taylor expansion of $C_\alpha$ about $1/2$:
\[
   C_\alpha=\sum_{\ell=0}^\infty C_\ell(\alpha-1/2)^\ell.
\]
The coefficients are compact by the Banach-valued Cauchy formula.  On a circle
$|\alpha-1/2|=R$ with $R>1/2$, Cauchy's estimate gives
\[
   \norm{C_\ell}\le M_RR^{-\ell},
   \qquad M_R=\max_{|\alpha-1/2|=R}\norm{C_\alpha},
\]
and consequently
\begin{equation}\label{eq:Taylor-summability-general}
   \sum_{\ell=0}^\infty2^{-\ell}\norm{C_\ell}<\infty.
\end{equation}
Let $J:\ell^2(I\times\{1,\ldots,m\})\to H^2(\C^m)$ map the coordinate
$(k,j)$ to $z^{n_k}e_j$.  If the multiplicity in
\eqref{eq:integer-multiplicity-general} is at most $M$, then
$\norm J\le\sqrt M$.  If $M_\ell$ is diagonal with entries
$(\alpha_k-1/2)^\ell$, then $\norm{M_\ell}\le2^{-\ell}$ and the synthesis
difference equals
\[
   \sum_{\ell=0}^\infty C_\ell JM_\ell.
\]
Each summand is compact, while
\[
 \sum_{\ell=0}^\infty
 \norm{C_\ell JM_\ell}
 \le \norm J\sum_{\ell=0}^\infty
        2^{-\ell}\norm{C_\ell}<\infty.
\]
Thus, the series converges in operator norm to a compact operator.
Finally, since $h\in\cN\subset H^2(\C^m)$ and $P^*h=h$, the Hardy-space
inner product with $PT_{q_{\alpha_k}}(z^{n_k}e_j)$ is exactly the integral in
\eqref{eq:localized-coefficient-general}.
\end{proof}

The following lemma establishes compactness of the restriction of a multiplication operator away from the essential spectrum.

\begin{lemma}
\label{lem:compact-restriction-general}
If $\varphi\in C(\Torus)$ vanishes on a neighborhood of $E_A$, then
\begin{equation}\label{eq:R-phi-general}
   R_\varphi:\cN\longrightarrow L^2(\Torus;\C^m),
   \qquad R_\varphi h=\varphi h,
\end{equation}
is compact.
\end{lemma}

\begin{proof}
The positive operator $R_\varphi^*R_\varphi$ equals
$A_{|\varphi|^2}$.  Its symbol vanishes on $E_A$, so
\cref{lem:continuous-symbol-calculus} makes it compact.  Hence
$R_\varphi$ is compact.
\end{proof}

The following form of Peetre's lemma converts an estimate with a compact
remainder into a closed-range statement; see \cite[Lemma~3]{Peetre61}.

\begin{lemma}
\label{lem:peetre-general}
Let $X,Y,Z$ be Hilbert spaces, let $B\in\mathcal B(X,Y)$, and let
$K\in\mathcal K(X,Z)$.  If
\begin{equation}\label{eq:peetre-general}
   \norm x_X\le C\norm{Bx}_Y+\norm{Kx}_Z,
   \qquad x\in X,
\end{equation}
then $\Ker B$ is finite-dimensional and $\Ran B$ is closed.  If, in addition,
$\Ker B=\{0\}$, then $B$ is bounded below.
\end{lemma}

\begin{proof}
On $\Ker B$, inequality \eqref{eq:peetre-general} says
$\norm x\le\norm{Kx}$.  The unit ball of $\Ker B$ is therefore relatively
compact, so the kernel is finite-dimensional.  If the range were not closed,
there would be unit vectors $x_n\perp\Ker B$ with $Bx_n\to0$.  Passing to a
subsequence, compactness gives $Kx_n\to z$.  Applying
\eqref{eq:peetre-general} to $x_n-x_m$ shows that $(x_n)$ is Cauchy.  Its limit
$x$ has norm one, is orthogonal to $\Ker B$, and satisfies $Bx=0$, a
contradiction.  Thus the range is closed.  When the kernel is trivial, the same
closed-range conclusion is equivalent to a lower bound for $B$.
\end{proof}

\begin{theorem}
\label{thm:fredholm-transfer-general}
Assume the setting of \cref{prop:basic-model,lem:compact-localization-general},
fix $d\in(c_{\mathrm e},\pi)$ and the cutoff $\chi=\chi_d$, and let
$\bm s=(s_k)_{k\in I}$ satisfy the bounded integer-part multiplicity condition
\eqref{eq:integer-multiplicity-general}.  Define
\[
  R_d:\mathcal N\longrightarrow L^2(-d,d;\mathbb C^m),
  \qquad R_dh=\chi h,
  \qquad \mathcal Y_d=R_d\mathcal N.
\]
Then $\mathcal Y_d$ is closed and
$R_d:\mathcal N\to\mathcal Y_d$ is boundedly invertible.

Let $D_{T,\bm s}$ denote the synthesis operator of
$\{T^{s_k}f_j\}_{k,j}$, and let $D_{\mathcal E_d(\bm s)}$ denote the synthesis
operator of the projected exponential family
\begin{equation}\label{eq:projected-exponential-family-general}
  \mathcal E_d(\bm s)
  =\left\{P_{\mathcal Y_d}(e^{is_kt}e_j):k\in I,\ 1\le j\le m\right\}
  \subset\mathcal Y_d.
\end{equation}
There exist a boundedly invertible operator
\[
  V_d=XR_d^*:\mathcal Y_d\longrightarrow\mathcal H
\]
and a compact operator
$K_{\bm s}:\ell^2(I\times\{1,\ldots,m\})\to\mathcal H$ such that
\begin{equation}\label{eq:fredholm-transfer-general}
  D_{T,\bm s}=V_dD_{\mathcal E_d(\bm s)}+K_{\bm s}.
\end{equation}
Consequently, $D_{T,\bm s}$ is upper semi-Fredholm, lower semi-Fredholm, or
Fredholm if and only if $D_{\mathcal E_d(\bm s)}$ has the corresponding
property.  In the Fredholm case,
\begin{equation}\label{eq:fredholm-index-general}
  \operatorname{ind}D_{T,\bm s}
  =\operatorname{ind}D_{\mathcal E_d(\bm s)}.
\end{equation}
If the projected exponential family is a frame for $\mathcal Y_d$ and the
dynamical family is complete in $\mathcal H$, then the dynamical family is a
frame.
\end{theorem}

\begin{proof}
Since $1-\chi$ vanishes on a neighborhood of $E_A$,
\cref{lem:compact-restriction-general} makes
$h\mapsto(1-\chi)h$ compact on $\mathcal N$.  By the triangle inequality
\[
  \|h\|\le \|R_dh\|+\|(1-\chi)h\|,
  \qquad h\in\mathcal N,
\]
and \cref{lem:peetre-general} applies to ensure that $R_d$ has closed range and finite-dimensional
kernel.  If $R_dh=0$, then every scalar component of $h$ has boundary values
zero on an open arc on which $\chi=1$.  F.~Riesz's Hardy-space uniqueness theorem gives $h=0$ \cite{Mash09}.
Thus $R_d$ is bounded below, $\mathcal Y_d$ is closed, and
$R_d:\mathcal N\to\mathcal Y_d$ is boundedly invertible.  Therefore
$R_d^*:\mathcal Y_d\to\mathcal N$ and $V_d=XR_d^*$ are boundedly invertible.

The bounded integer-part multiplicity assumption implies
\[
   \sup_{x\in\mathbb R}
   \#\{k:s_k\in[x,x+1]\}<\infty.
\]
Hence $\bm s$ is relatively separated, and the standard upper
nonharmonic Fourier estimate shows that
$\{e^{is_kt}e_j\}_{k,j}$ is Bessel on $L^2(-d,d;\mathbb C^m)$.
Its orthogonal projection onto $\mathcal Y_d$ is therefore Bessel, so
$D_{\mathcal E_d(\bm s)}$ is bounded.  The compact comparison from
Lemma~\ref{lem:compact-localization-general} then also makes
$D_{T,\bm s}$ bounded.

Write $s_k=n_k+\alpha_k$ as in
\cref{lem:compact-localization-general}.  For
$u_{k,j}=P_{\mathcal Y_d}(e^{is_kt}e_j)$, one has
\[
  R_d^*u_{k,j}
  =PT_{q_{\alpha_k}}(z^{n_k}e_j)=w_{k,j}.
\]
Indeed, this identity follows by testing against $h\in\mathcal N$ and using
\eqref{eq:localized-coefficient-general}.  Thus the synthesis operator of the
localized columns $(w_{k,j})$ is $R_d^*D_{\mathcal E_d(\bm s)}$.
By \cref{lem:compact-localization-general}, the synthesis operator of
$(A^{s_k}g_j)$ differs from this operator by a compact map.  Multiplication by
the similarity $X$ and the identity $T^{s_k}f_j=XA^{s_k}g_j$ give
\eqref{eq:fredholm-transfer-general}.

Semi-Fredholmness and Fredholmness are invariant under invertible left
multiplication and compact perturbations, and the Fredholm index is unchanged.
Finally, if $D_{\mathcal E_d(\bm s)}$ is onto, then
\eqref{eq:fredholm-transfer-general} makes $D_{T,\bm s}$ lower semi-Fredholm.
Its range is therefore closed and of finite codimension.  Completeness makes
that range dense, hence equal to $\mathcal H$.
\end{proof}

\begin{remark}
\label{rem:beyond-sectorial-localization}
The compact-transfer mechanism does not itself require a spectral sector.  Let
$E_A=\sigma_{\mathrm e}(A)\subset\mathbb T$, let $\varphi$ be holomorphic near
$\sigma(A)$, and choose $q\in C(\mathbb T)$ with
$q|_{E_A}=\varphi|_{E_A}$.  Then
\begin{equation}\label{eq:general-holomorphic-localization}
   \varphi(A)P-PT_q
   \in \mathcal K\bigl(H^2(\mathbb C^m),\mathcal N\bigr).
\end{equation}
Indeed, $\varphi(\dot A)=q(\dot A)$, and the remaining Toeplitz-compression
error is compact by \cref{lem:continuous-symbol-calculus}.  The same argument,
together with the compact-valued Taylor expansion used above, applies to
norm-holomorphic symbol families and therefore transfers semi-Fredholm
properties and Fredholm index in that setting as well.

Sectoriality is used to choose one logarithmic branch and turn the boundary
symbols into ordinary exponentials on a single arc.  A finite union of proper
arcs leads instead to a multiband exponential model.  If $E_A=\mathbb T$, the
global Toeplitz transfer \eqref{eq:general-holomorphic-localization} remains
valid, but the coordinate function has no continuous logarithm on the circle;
a coherent model for all noninteger powers then requires additional, generally
operator-valued, essential-symbol structure.
\end{remark}

\begin{remark}
\label{rem:frame-sequence-transfer-general}
A Bessel family is a frame sequence of finite excess precisely when its
synthesis operator is upper semi-Fredholm.  Hence
\cref{thm:fredholm-transfer-general} gives an equivalence of this property for
the dynamical and projected-exponential systems.  If both systems are complete
and one is a frame of finite excess, then both synthesis operators are
Fredholm; their excesses agree by \eqref{eq:fredholm-index-general}.
The projection onto $\mathcal Y_d$ is essential here: compact perturbations do
not preserve an arbitrary frame-sequence lower bound without a semi-Fredholm
hypothesis.
\end{remark}
\section{Completeness, complement properties, and separated sampling}
\label{sec:separated}

For a locally finite set $\Lambda\subset[0,\infty)$, recall its one-sided density $D_+^-(\Lambda)$ from
\eqref{eq:intro-density-definition} and the accompanying notion of uniform
separation.  We also write
\begin{equation}\label{eq:lower-natural-density-general}
  \underline d_{\rm nat}(\Lambda)
  =\liminf_{R\to\infty}
    \frac{\#(\Lambda\cap[0,R])}{R}
\end{equation}
for its lower natural density.

For a locally finite set $\Lambda\subset\R$, we will also use the two-sided lower Beurling density
\begin{equation}\label{eq:two-sided-density-general}
 D^-(\Lambda)=\liminf_{R\to\infty}\ \inf_{x\in\R}
 \frac{\#(\Lambda\cap[x,x+R])}{R}
\end{equation}
and Rubel's logarithmic block density
\begin{equation}\label{eq:log-block-density-general}
   L(\Lambda)=\inf_{\rho>1}\frac1{\log\rho}
   \limsup_{x\to\infty}
   \sum_{\lambda\in\Lambda\cap[x,\rho x]}\frac1\lambda.
\end{equation}
A finite number of points near zero are omitted from the sum and do not affect the value.
\subsection{Density comparisons and sectorial uniqueness}

To use the classical Beurling-density sampling theory, we augment the sets $\Lambda\subset[0,\infty)$ with negative integers.

\begin{lemma}
\label{lem:augmented-density-general}
For a uniformly separated set $\Lambda\subset[0,\infty)$, put
\[
   \widetilde\Lambda=\Lambda\cup\{-1,-2,-3,\ldots\}.
\]
Then $\widetilde\Lambda$ is uniformly separated and
\begin{equation}\label{eq:augmented-density-general}
   D^-(\widetilde\Lambda)=\min\{1,D_+^-(\Lambda)\}.
\end{equation}
\end{lemma}

\begin{proof}
Write $D=D_+^-(\Lambda)$.  The upper bound in
\eqref{eq:augmented-density-general} follows in two ways: intervals contained
far in the negative half-line have asymptotic density one, while intervals in
$[0,\infty)$ realizing the infimum in the definition of $D_+^-$ show that the
two-sided lower density is at most $D$.

For the reverse inequality, fix $\varepsilon>0$.  There is $R_0$ such that every
interval $[y,y+H]\subset[0,\infty)$ with $H\ge R_0$ contains at least
$(D-\varepsilon)H$ points of $\Lambda$.  Let $I=[x,x+R]$ be arbitrary and
split it at the origin into a negative part of length $R_-$ and a nonnegative
part of length $R_+$, so $R_-+R_+=R$.  The negative integers contribute at
least $R_--2$ points.  If $R_+\ge R_0$, the positive part contributes at least
$(D-\varepsilon)R_+$ points; if $R_+<R_0$, we simply discard that contribution.
In either case,
\[
 \#(\widetilde\Lambda\cap I)
 \ge \min\{1,D-\varepsilon\}R-(R_0+2).
\]
Taking the infimum in $x$, then the lower limit as $R\to\infty$, and finally
letting $\varepsilon\downarrow0$ proves the lower bound.  The augmented set is
uniformly separated because the two pieces are separated from one another by the
gap between $-1$ and $[0,\infty)$.
\end{proof}

The next lemma shows that logarithmic block density dominates both natural
and lower uniform density.  The proof also applies verbatim to multisets when
multiplicities are counted. The lemma follows from the results in \cite{Grekos05, Kuzhaev19, MalliavinRubel61}; we record the proof for completeness.

\begin{lemma}
\label{lem:uniform-log-general}
For every locally finite $\Lambda\subset[0,\infty)$,
\begin{equation}\label{eq:uniform-log-ineq-general}
   L(\Lambda)\ge \underline d_{\rm nat}(\Lambda)
   \ge D_+^-(\Lambda).
\end{equation}
The same inequalities hold for a locally finite multiset when all counting
functions and sums count multiplicity.
\end{lemma}

\begin{proof}
The second inequality follows by restricting the infimum in the definition of
$D_+^-$ to intervals beginning at the origin.  To prove the first, let
\[
  N(t)=\#(\Lambda\cap(0,t]),
  \qquad d=\underline d_{\rm nat}(\Lambda).
\]
If $d=\infty$, fix $\rho>1$.  If the logarithmic block sums on
$(x,\rho x]$ had finite upper limit, then for all sufficiently large $x$,
\[
  N(\rho x)-N(x)\le C\rho x
\]
for some $C>0$, because every point of $(x,\rho x]$ contributes at least
$1/(\rho x)$.  Iterating this estimate along $x,\rho x,\rho^2x,\ldots$ would
make $N(\rho^nx)/(\rho^nx)$ bounded, contradicting $d=\infty$.  Hence every
logarithmic-block upper limit is infinite and $L(\Lambda)=\infty$.

Assume now that $d<\infty$ and fix $\rho>1$.  Stieltjes integration by parts gives
\begin{equation}\label{eq:stieltjes-log-density-general}
 \sum_{\lambda\in\Lambda\cap(x,\rho x]}\frac1\lambda
 =\frac{N(\rho x)}{\rho x}-\frac{N(x)}x
   +\int_x^{\rho x}\frac{N(t)}{t^2}\,\dd t.
\end{equation}
Given $\varepsilon>0$, one has $N(t)/t\ge d-\varepsilon$ for all sufficiently
large $t$, and there are arbitrarily large $x$ for which
$N(x)/x\le d+\varepsilon$.  Along such a sequence, the right-hand side of
\eqref{eq:stieltjes-log-density-general} is bounded below by
\(
  (d-\varepsilon)\log\rho-2\varepsilon.
\)
Taking the upper limit in $x$, then letting $\varepsilon\downarrow0$, yields
\[
  \limsup_{x\to\infty}
  \sum_{\lambda\in\Lambda\cap(x,\rho x]}\frac1\lambda
  \ge d\log\rho.
\]
Dividing by $\log\rho$ and taking the infimum over $\rho>1$ establishes the result.  The multiset version
uses the same counting function with multiplicities.
\end{proof}

The next lemma gives a uniqueness criterion in terms of growth on the
imaginary axis and the density of positive zeros.  It is a consequence of
the Carlson--Rubel theorem \cite{Rubel56,MalliavinRubel61}.

\begin{lemma}
\label{lem:carlson-rubel-general}
Let $F$ be an entire function of exponential type and let $a\ge0$.
Suppose that, for every $\varepsilon>0$,
\begin{equation}\label{eq:rubel-growth-general}
   |F(iy)|\le C_\varepsilon e^{(a+\varepsilon)|y|},
   \qquad y\in\R.
\end{equation}
If $F\not\equiv0$ and its zero divisor contains a locally finite
multisequence $\Lambda\subset(0,\infty)$, then
\begin{equation}\label{eq:rubel-bound-general}
   L_{\rm mult}(\Lambda)\le a/\pi,
\end{equation}
where $L_{\rm mult}$ denotes \eqref{eq:log-block-density-general} with
multiplicities counted.  For a set of distinct zeros this gives
$L(\Lambda)\le a/\pi$.
\end{lemma}

\begin{proof}
The assertion is immediate for finite $\Lambda$.  Define the
\emph{indicator function} of $F$ (see \cite[Chapter~VI]{Levin80}) by
\[
 h_F(\theta)=\limsup_{r\to\infty}
       \frac{\log|F(re^{i\theta})|}{r},
 \qquad \theta\in\R.
\]
The growth assumption \eqref{eq:rubel-growth-general} gives
\[
 h_F(\pi/2)\le a,\qquad h_F(-\pi/2)\le a.
\]

Let $N_F(R)$ denote the number of zeros of $F$ in $|z|\le R$,
counted with multiplicity.  Jensen's formula gives $N_F(R)=O(R)$.
Since $\#(\Lambda\cap(0,R])\le N_F(R)$, the multisequence $\Lambda$
has finite upper natural density.
The necessity direction of \cite[Theorem~6.2]{MalliavinRubel61}
applies to such multisequences and states that
$L_{\rm mult}(\Lambda)\le b$ whenever
$h_F(\pm\pi/2)\le\pi b$.  Taking $b=a/\pi$ proves
\eqref{eq:rubel-bound-general}.
\end{proof}

\subsection{Completeness and the complement property}

We apply the uniqueness criterion to the entire coefficient functions
$z\mapsto\langle T^zf_j,h\rangle$.  This yields a completeness condition
determined by the full spectral sector.

\begin{theorem}
\label{thm:sectorial-completeness-general}
Let $T\in\mathcal B(\Hh)$ be invertible, assume
$\sigma(T)\subset\widehat\Sigma_c$ for some $0\le c<\pi$, and suppose that
\(
  \{T^nf_j:n\in\Nzero,\ 1\le j\le m\}
\)
is complete in $\Hh$.  For each channel let
$\Lambda_j\subset[0,\infty)$ be locally finite.  As in the definition of
$L$, a possible point at zero is ignored.  If
\begin{equation}\label{eq:log-completeness-condition-general}
   L(\Lambda_j)>\frac c\pi,
   \qquad 1\le j\le m,
\end{equation}
then
\(
  \{T^\lambda f_j:\lambda\in\Lambda_j,\ 1\le j\le m\}
\)
is complete in $\Hh$.
\end{theorem}

\begin{proof}
Let $h$ be orthogonal to the sampled family and set
\[
   F_{h,j}(z)=\langle T^zf_j,h\rangle,
   \qquad z\in\C.
\]
Since $\Log T$ is bounded, $F_{h,j}$ is entire of exponential type.  The
spectral inclusion gives
\[
   \sigma(\Log T)\subset
   \{u+iv:u\in\R,\ |v|\le c\}.
\]
Applying the spectral-radius formula to $e^{i\Log T}$ and $e^{-i\Log T}$,
and splitting $|y|$ into an integer and a bounded remainder, shows that for
every $\varepsilon>0$,
\[
  \|e^{iy\Log T}\|
  \le C_\varepsilon e^{(c+\varepsilon)|y|},
  \qquad y\in\R.
\]
Thus $F_{h,j}$ satisfies the imaginary-axis estimate in
\cref{lem:carlson-rubel-general} with $a=c$ and vanishes on $\Lambda_j$.
The density condition \eqref{eq:log-completeness-condition-general}
therefore gives $F_{h,j}\equiv0$.
Hence
\[
  \langle T^nf_j,h\rangle=0,
  \qquad n\in\Nzero,\ 1\le j\le m.
\]
Completeness of the integer multiorbit gives $h=0$.
\end{proof}

\begin{corollary}
\label{cor:natural-density-completeness-general}
In the hypotheses of \cref{thm:sectorial-completeness-general}, one may replace \eqref{eq:log-completeness-condition-general} with
\[
  \underline d_{\rm nat}(\Lambda_j)> c/\pi \quad \mbox{ or } \quad D_+^-(\Lambda_j)>c/\pi,
  \qquad 1\le j\le m.
\]
\end{corollary}

\begin{proof}
Apply \cref{lem:uniform-log-general,thm:sectorial-completeness-general}.
\end{proof}

\begin{remark}
\label{rem:BM-refinement}
There is a Beurling--Malliavin sufficient criterion for complete orbits
under the additional hypothesis
\begin{equation}\label{eq:bounded-real-power-group}
  \sup_{t\in\mathbb R}\|T^t\|<\infty.
\end{equation}
For a discrete set $\Lambda\subset\mathbb R$, put
\[
\begin{gathered}
 R(\Lambda)=\sup\left\{a>0:
   \closspan\{e^{i\lambda t}:\lambda\in\Lambda\}=L^2(-a,a)\right\},\\
 D_{\rm BM}(\Lambda)=R(\Lambda)/\pi.
\end{gathered}
\]
This is the exterior Beurling--Malliavin density in the stated
normalization \cite{BM67,MNK06}.  Write $PW_a$ for the Paley--Wiener
space of entire functions of exponential type at most $a$ whose
restrictions to $\mathbb R$ belong to $L^2(\mathbb R)$.

Suppose that $\sigma(T)\subset\widehat\Sigma_c$ and the integer
multiorbit is complete.  Under \eqref{eq:bounded-real-power-group},
$\sigma(\Log T)\subset i[-c,c]$: boundedness for both signs of $t$
forces the real parts of its spectral points to vanish.  The
spectral-radius formula for powers of $\Log T$ then shows that each
$F_{h,j}(z)=\langle T^zf_j,h\rangle$ has exponential type at most $c$.
It is also bounded on the real axis.  If
$D_{\rm BM}(\Lambda_j)>c/\pi$, choose $\varepsilon>0$ such that
$c+\varepsilon<R(\Lambda_j)$.  A nonzero coefficient function
vanishing on $\Lambda_j$ would give a nonzero function
\[
  F_{h,j}(z)\left(\frac{\sin(\varepsilon z/2)}{z}\right)^2
  \in PW_{c+\varepsilon}
\]
vanishing on $\Lambda_j$, with the removable value at zero.
Paley--Wiener duality contradicts the definition of $R(\Lambda_j)$.
Thus the sampled multiorbit is complete.  Applying the same argument
to products of coefficient functions replaces the density hypothesis
in \cref{thm:complement-property-general} by
$D_{\rm BM}(\Lambda)>2c/\pi$.

Under \eqref{eq:bounded-real-power-group}, a unilateral multiorbit can
be a frame only on the zero space.  Indeed, if it were a frame with lower
bound $a$, then
\[
 a\|T^{*n}h\|^2
 \le\sum_{k\ge n}\sum_{j=1}^m
       |\langle h,T^kf_j\rangle|^2\longrightarrow0.
\]
Uniform boundedness of $T^{-n}$ would imply
$\|h\|\le C\|T^{*n}h\|\to0$, so $\mathcal H=\{0\}$.
\end{remark}

For $J\subset\{1,\ldots,m\}$, define the channel-cyclic subspace
\begin{equation}\label{eq:channel-cyclic-subspace-general}
  \Hh_J=\overline{\operatorname{span}}
  \{T^nf_j:n\in\Nzero,\ j\in J\}.
\end{equation}

In the following proposition, we establish that fractional powers do not enlarge channel-cyclic subspaces.

\begin{proposition}
\label{prop:fractional-channel-cyclic-general}
Let $T$ be invertible with
$\sigma(T)\subset\widehat\Sigma_c$, $c<\pi$.  For every
$J\subset\{1,\ldots,m\}$,
\begin{equation}\label{eq:fractional-channel-equality-general}
 \overline{\operatorname{span}}
 \{T^sf_j:s\ge0,\ j\in J\}
 =\overline{\operatorname{span}}
 \{T^nf_j:n\in\Nzero,\ j\in J\}=\Hh_J.
\end{equation}
\end{proposition}

\begin{proof}
Only the inclusion from left to right requires proof.  Let $x\perp\Hh_J$.
For $j\in J$, the function
$F_{x,j}(z)=\langle T^zf_j,x\rangle$ vanishes on $\Nzero$.  Since
$L(\Nzero)=1>c/\pi$, \cref{lem:carlson-rubel-general} and the growth argument
from \cref{thm:sectorial-completeness-general} give
$F_{x,j}\equiv0$.  Hence $x$ annihilates every $T^sf_j$, $s\ge0$.
Taking orthogonal complements proves the assertion.
\end{proof}

A family $(\phi_i)_{i\in I}$ has the \emph{complement property} if, for every
$E\subset I$, either $(\phi_i)_{i\in E}$ or $(\phi_i)_{i\in I\setminus E}$ is
complete. This property is useful in the study of the phase-retrieval problem \cite{BCE06}.
The following theorem provides a bound on logarithmic density, which yields a useful criterion for the complement property.

\begin{theorem}
\label{thm:complement-property-general}
Let $T\in\mathcal B(\Hh)$ be invertible, assume
$\sigma(T)\subset\widehat\Sigma_c$, and let a common locally finite set
$\Lambda\subset[0,\infty)$ satisfy
\begin{equation}\label{eq:cp-density-general}
  L(\Lambda)>\frac{2c}{\pi}.
\end{equation}
Then the sampled family
\(
  \{T^\lambda f_j:\lambda\in\Lambda,\ 1\le j\le m\}
\)
has the complement property if and only if
\begin{equation}\label{eq:channel-complement-general}
  \Hh_J=\Hh\quad\hbox{or}\quad \Hh_{J^c}=\Hh
  \qquad\text{for every }J\subset\{1,\ldots,m\}.
\end{equation}
\end{theorem}

\begin{proof}
If \eqref{eq:channel-complement-general} fails, partition the sampled family
according to the channels in $J$ and $J^c$.  By
\cref{prop:fractional-channel-cyclic-general}, the two resulting spans are
contained in the two proper subspaces $\Hh_J$ and $\Hh_{J^c}$, so the
complement property fails.

Conversely, suppose a partition $E\sqcup E^c$ of the sampled index set makes
both sides incomplete.  Choose nonzero $x,y\in\Hh$ such that $x$ is
orthogonal to the $E$-side and $y$ is orthogonal to the $E^c$-side.  Put
\[
  F_{x,j}(z)=\langle T^zf_j,x\rangle,
  \qquad
  F_{y,j}(z)=\langle T^zf_j,y\rangle.
\]
For every $\lambda\in\Lambda$ and every $j$, one of these two values is zero;
hence
\[
  G_j(z)=F_{x,j}(z)F_{y,j}(z)
\]
vanishes on $\Lambda$.  The growth estimate in the proof of
\cref{thm:sectorial-completeness-general} shows that $G_j$ is entire
of exponential type and has imaginary-axis type at most $2c$.  By
\eqref{eq:cp-density-general} and \cref{lem:carlson-rubel-general},
$G_j\equiv0$.  Since the ring of entire functions has no zero divisors, for
each $j$ either $F_{x,j}\equiv0$ or $F_{y,j}\equiv0$.

Let $J=\{j:F_{x,j}\equiv0\}$.  Then $x\perp\Hh_J$ and
$y\perp\Hh_{J^c}$.  Because $x$ and $y$ are nonzero, both channel-cyclic
subspaces are proper, contradicting \eqref{eq:channel-complement-general}.
\end{proof}

\begin{corollary}
\label{cor:single-complement-general}
Suppose that $\{T^nf:n\in\Nzero\}$ is complete and that the operator
hypotheses of \cref{thm:sectorial-completeness-general} hold.  If
$L(\Lambda)>2c/\pi$, then $\{T^\lambda f:\lambda\in\Lambda\}$ has the
complement property.
\end{corollary}

\begin{proof}
For one channel, \eqref{eq:channel-complement-general} is automatic.
\end{proof}

\subsection{Stable sampling at separated times}

Combining compact Fredholm localization with the completeness criterion
gives a frame theorem with separate density conditions for closed range
and completeness.

\begin{theorem}
\label{thm:separated-general}
Assume that \eqref{eq:intro-frame} is a frame, $T$ is invertible,
and \eqref{eq:intro-sector} and \eqref{eq:essential-sector-general} hold.
Let $\Lambda\subset[0,\infty)$ be uniformly separated.  If
\begin{equation}\label{eq:separated-essential-density-main}
   D_+^-(\Lambda)>c_{\mathrm e}/\pi,
\end{equation}
then the family
\begin{equation}\label{eq:separated-family-main}
   \{T^\lambda f_j:\lambda\in\Lambda,\ 1\le j\le m\}
\end{equation}
is Bessel and its synthesis operator has closed range of finite
codimension.  If, in addition,
\begin{equation}\label{eq:separated-log-density-main}
   L(\Lambda)>c/\pi,
\end{equation}
then the family is a frame.  In particular, the single condition
\begin{equation}\label{eq:separated-density-main}
   D_+^-(\Lambda)>c/\pi
\end{equation}
suffices.

The same assertions hold for channel-dependent uniformly separated sets
$\Lambda_j$: the synthesis operator of
$\{T^\lambda f_j:\lambda\in\Lambda_j,\ 1\le j\le m\}$ has closed
range of finite codimension if $D_+^-(\Lambda_j)>c_{\mathrm e}/\pi$
for every $j$, and the family is a frame if also $L(\Lambda_j)>c/\pi$
for every $j$.
\end{theorem}

\begin{proof}
We first consider a common temporal set $\Lambda$.  Since $\Lambda$ is
uniformly separated, the integer parts of its elements have uniformly bounded
multiplicity, so \cref{thm:fredholm-transfer-general} applies.

Choose
\(
   c_{\mathrm e}<d<\pi\min\{1,D_+^-(\Lambda)\},
\)
and let $\mathcal Y_d$ be the localized Hardy-model space appearing in
\cref{thm:fredholm-transfer-general}.  Put
\(
   \widetilde\Lambda
   =
   \Lambda\cup\{-1,-2,-3,\ldots\}.
\)
By \cref{lem:augmented-density-general},
\[
   D^-(\widetilde\Lambda)
   =
   \min\{1,D_+^-(\Lambda)\}
   >
   \frac d\pi.
\]
Beurling's sampling theorem \cite{Beurling89,Young01} therefore implies that
\[
   \{e^{i\mu t}e_j:
     \mu\in\widetilde\Lambda,\ 1\le j\le m\}
\]
is a frame for $L^2(-d,d;\C^m)$.  Its orthogonal projection onto
$\mathcal Y_d$ is consequently a frame for $\mathcal Y_d$.

Let us write this projected frame as the union of
\[
   \mathcal E_d^+(\Lambda)
   =
   \left\{
      P_{\mathcal Y_d}(e^{i\lambda t}e_j):
      \lambda\in\Lambda,\ 1\le j\le m
   \right\}
\]
and
\[
   \mathcal E_d^-
   =
   \left\{
      P_{\mathcal Y_d}(e^{-iqt}e_j):
      q\ge1,\ 1\le j\le m
   \right\}.
\]
The analysis operator of $\mathcal E_d^-$ is compact.  Indeed, if
$u=R_dh=\chi h\in\mathcal Y_d$, then, since every component of
$h\in H^2(\C^m)$ has vanishing negative Fourier coefficients,
\[
\begin{aligned}
   \left\langle
      u,P_{\mathcal Y_d}(e^{-iqt}e_j)
   \right\rangle
   &=
   \left\langle
      \chi h,e^{-iqt}e_j
   \right\rangle                                      \\
   &=
   \left\langle
      (\chi-1)h,e^{-iqt}e_j
   \right\rangle .
\end{aligned}
\]
The map
\[
   h\longmapsto(\chi-1)h
\]
is compact by \cref{lem:compact-restriction-general}, while Fourier analysis
from $L^2(\Torus;\C^m)$ into $\ell^2(\N\times\{1,\ldots,m\})$ is bounded.
Since $R_d^{-1}:\mathcal Y_d\to\mathcal N$ is bounded, the asserted
compactness follows.

Let $D_{\rm pos}$ and $D_{\rm neg}$ be the synthesis operators of
$\mathcal E_d^+(\Lambda)$ and $\mathcal E_d^-$, respectively.  The operator
\(
   [\,D_{\rm pos}\ \ D_{\rm neg}\,]
\)
is onto $\mathcal Y_d$, whereas $D_{\rm neg}$ is compact as the adjoint of the compact analysis operator of $\mathcal E_d^-$.  Hence,
\(
   [\,D_{\rm pos}\ \ 0\,]
\)
is a compact perturbation of an onto operator and is therefore lower
semi-Fredholm.  Its range is exactly $\Ran D_{\rm pos}$, so
$D_{\rm pos}$ itself has closed range of finite codimension.

By \cref{thm:fredholm-transfer-general}, the synthesis operator of
\[
   \{T^\lambda f_j:
     \lambda\in\Lambda,\ 1\le j\le m\}
\]
is likewise lower semi-Fredholm.  This proves the first assertion.
If \eqref{eq:separated-log-density-main} also holds, then
\cref{thm:sectorial-completeness-general} makes the sampled family
complete.  Its synthesis range is therefore both dense and closed,
and hence equals $\Hh$.  Finally, \eqref{eq:separated-density-main}
implies both required inequalities by
$c_{\mathrm e}\le c$ and \cref{lem:uniform-log-general}.

For channel-dependent sets $\Lambda_j$, we choose
\[
   c_{\mathrm e}<d<
   \pi\min_{1\le j\le m}\min\{1,D_+^-(\Lambda_j)\}
\]
and apply Beurling's theorem in the $j$th coordinate to
\(
   \Lambda_j\cup\{-1,-2,\ldots\}.
\)
The negative-frequency analysis operators are compact componentwise, and
the proof of \cref{thm:fredholm-transfer-general} applies without change to
the channel-indexed set
\(
   \bigsqcup_{j=1}^m(\Lambda_j\times\{j\}).
\)
Thus, the corresponding dynamical synthesis operator is lower
semi-Fredholm.  If $L(\Lambda_j)>c/\pi$ for every channel,
\cref{thm:sectorial-completeness-general} gives completeness and
therefore surjectivity.
\end{proof}

\begin{remark}
The compact comparison of \cref{thm:fredholm-transfer-general} extends
to an arbitrary channel-indexed family
$\{T^{s_k}f_{j_k}:k\in I\}$ of nonnegative times, provided
\[
 \sup_{n\in\mathbb Z}\#\{k:\lfloor s_k\rfloor=n\}<\infty.
\]
Indeed, the proof is columnwise and uses only bounded multiplicity;
equivalently, restrict the all-channel comparison to the corresponding
coordinate subspace.  This condition ensures Besselness and the compact
identity.  A frame conclusion follows when this comparison is combined
with lower semi-Fredholmness and completeness, as in
\cref{thm:separated-general}.
\end{remark}

\begin{corollary}
\label{cor:progressions-general}
For every $r>0$ satisfying
\begin{equation}\label{eq:progression-range-general}
   rc<\pi,
\end{equation}
the family
\begin{equation}\label{eq:progression-family-general}
   \{T^{rk}f_j:k\in\Nzero,\ 1\le j\le m\}
\end{equation}
is a frame.  In particular, every $r>0$ is allowed when $c=0$.
\end{corollary}

\begin{proof}
The set $r\Nzero$ is uniformly separated and has lower uniform density $1/r$.
Apply \cref{thm:separated-general}.
\end{proof}

\begin{remark}
\label{rem:bdc-comparison-general}
When $c=0$, \cref{cor:natural-density-completeness-general} first gives
completeness under positive lower natural density, while
\cref{thm:separated-general} gives a universal sufficient frame condition for arbitrary invertible finite-generator dynamical frames with
positive spectrum: every separated temporal set of positive lower uniform
density yields a frame.  Theorem 1.3 in \cite{KMBDC26} is sharper for its
special block-diagonal Carleson (BDC) class.  Assuming that the natural density
exists, it characterizes the frame property by positivity and finiteness of
that density and does not require uniform local density.  The BDC result
uses the additional interpolation and Jordan-block structure to replace a
local sampling condition by a global asymptotic rate, while the theorem here
applies to general, possibly nonnormal, finite-multiplicity Hardy models.

For positive selfadjoint operators, Wu
\cite[Theorem~4.4 and Corollary~6.4]{WuSampling26} obtains a universal
natural-density characterization that allows countably many generators
and noninvertible operators: when the natural density exists, it must be
finite and positive, and the temporal set must contain zero.  This
positive-operator result should be distinguished from the case $c=0$
here, which allows nonnormal operators with positive spectrum.

\end{remark}

In the remainder of the section, we discuss the complement property and phase retrieval.

\begin{corollary}
\label{cor:narrow-sector-complement-general}
Under the standing frame assumptions, if $c<\pi/2$, then the integer
dynamical frame has the complement property if and only if the channel-cyclic
condition \eqref{eq:channel-complement-general} holds.  More generally, for
$r>0$ with $rc<\pi/2$, the progression frame
\(
  \{T^{rk}f_j:k\in\Nzero,\ 1\le j\le m\}
\)
has the complement property if and only if
\eqref{eq:channel-complement-general} holds.  In particular, every singly
generated sectorial dynamical frame with $c<\pi/2$ has the complement property.
\end{corollary}

\begin{proof}
Apply \cref{thm:complement-property-general} first with
$\Lambda=\Nzero$, for which $L(\Lambda)=1$, and then with
$\Lambda=r\Nzero$, for which $L(\Lambda)=1/r$.  The frame assertion for the
progression follows from \cref{cor:progressions-general}.
\end{proof}

\begin{corollary}
\label{cor:real-phase-retrieval-general}
Let $\Hh_{\mathbb R}$ be a real Hilbert space, let $T$ and
$f_1,\ldots,f_m$ be real, and assume that the complexification of $T$
satisfies the operator and frame hypotheses of
\cref{thm:separated-general}.  Let a common temporal set $\Lambda$ be
uniformly separated and satisfy
\begin{equation}\label{eq:real-phase-density-general}
  D_+^-(\Lambda)>{2c}/{\pi}.
\end{equation}
If the channel-cyclic subspaces satisfy
\eqref{eq:channel-complement-general}, then
\(
  \{T^\lambda f_j:\lambda\in\Lambda,\ 1\le j\le m\}
\)
does phase retrieval for $\Hh_{\mathbb R}$.
For a singly generated frame,
\(
  \{T^{rk}f:k\in\Nzero\}
\)
does real phase retrieval whenever $rc<\pi/2$.
\end{corollary}

\begin{proof}
The principal logarithm is intrinsic on its domain:
$\Log\overline z=\overline{\Log z}$.  Since the complexification of
$T$ commutes with conjugation, its holomorphic functional calculus gives
the same property for $\Log T$ and hence for every real power $T^s$.
Thus the sampled vectors are real.  By
\cref{thm:separated-general}, the sampled family is a frame.  By
\cref{lem:uniform-log-general,thm:complement-property-general}, it has the
complement property.  Over a real Hilbert space, a frame performs phase
retrieval if and only if it has the complement property; see
\cite{BCE06,CCD16}.
\end{proof}

\begin{remark}
\label{rem:phase-retrieval-history-general}
The complement-property formulation of real phase retrieval originates in
\cite{BCE06}; the infinite-dimensional setting was developed in \cite{CCD16},
and dynamical phase retrieval in \cite{AKT17,AKT20}.  For a real diagonal
Carleson frame $\{D^ng\}_{n\ge0}$ in the sense of \cite{CHPS24}, the results of \cite{Fard26}, combined with
\cite[Theorem~3.10]{AKT20}, give the exact criterion that the eigenvalues of
$D$ be nonzero and have pairwise distinct moduli.  This uses substantially
more diagonal spectral structure than the universal sectorial theorem above.
Over complex Hilbert spaces the complement property is necessary but not
sufficient for phase retrieval.
\end{remark}

\section{Bounded temporal perturbations}

\label{sec:perturbations}

Throughout this section, the integer multiorbit
\eqref{eq:intro-frame} is a frame, $T$ is invertible, and
\eqref{eq:intro-sector} and \eqref{eq:essential-sector-general} hold.
We extract separated residue classes from bounded temporal perturbations.

\begin{theorem}
\label{thm:cluster-extraction-general}
Let $r>0$, $s_0\in\R$, and
\begin{equation}\label{eq:cluster-perturbation-general}
   \lambda_k=s_0+rk+\delta_k,
   \qquad
   \Delta=\sup_{k\ge0}|\delta_k|<\infty.
\end{equation}
Set
\begin{equation}\label{eq:cluster-number-general}
   q=q(\Delta,r)
   :=\left\lfloor\frac{2\Delta}{r}\right\rfloor+1.
\end{equation}
If
\begin{equation}\label{eq:cluster-condition-general}
   qrc<\pi,
\end{equation}
then
\begin{equation}\label{eq:cluster-family-general}
   \{T^{\lambda_k}f_j:k\in\Nzero,\ 1\le j\le m\}
\end{equation}
is a frame.
\end{theorem}

\begin{proof}
For $\ell=0,\ldots,q-1$, consider the residue-class times
\[
   \Lambda_\ell
   =\{\lambda_{qn+\ell}:n\in\Nzero\}.
\]
Since $q>2\Delta/r$,
\[
\begin{aligned}
 \lambda_{q(n+1)+\ell}-\lambda_{qn+\ell}
 &=qr+\delta_{q(n+1)+\ell}-\delta_{qn+\ell}\\
 &\ge qr-2\Delta>0.
\end{aligned}
\]
Thus, each $\Lambda_\ell$ is uniformly separated.  Because the deviations are
bounded, its counting function differs uniformly by $O(1)$ from that of an
arithmetic progression of step $qr$.  Deleting its finitely many negative
points therefore leaves a set $\Lambda_\ell^+\subset[0,\infty)$ with
\(
   D_+^-(\Lambda_\ell^+)=1/{qr}> c/\pi.
\)
By \cref{thm:separated-general},
\(
   \{T^\lambda f_j:\lambda\in\Lambda_\ell^+,\ 1\le j\le m\}
\)
is a frame.  Adding the finitely many omitted negative-time vectors preserves
the frame property.  Hence, every residue-class family is a frame, and their
finite union \eqref{eq:cluster-family-general} is a frame.
\end{proof}

\begin{remark}
\label{rem:kadec-avdonin-general}
The above cluster-extraction theorem contains the qualitative Kadec $1/4$ regime.
Indeed, if $rc<\pi$ and $c\Delta<\pi/4$, then either $\Delta<r/2$ and $q=1$,
or
\[
   qrc\le 2c\Delta+cr\le4c\Delta<\pi;
\]
hence \cref{thm:cluster-extraction-general} applies.  The same pointwise
threshold occurs for sampled isometric group orbits in \cite{KM25}.

For separated bounded perturbations, Avdonin's theorem gives a classical
block-average formulation: after normalizing
$\lambda_k=s_0+rk+\delta_k$ by a constant offset, a bounded two-sided extension
of $(\delta_k/r)$ whose frequencies remain separated and whose averages over
some fixed block length stay strictly below $1/4$ generates a Riesz basis of
exponentials; see \cite{Avdonin74,Young01}.  In the present paper the resulting
dynamical-frame conclusion already follows from
\cref{thm:separated-general}, since a separated bounded perturbation has density
$1/r$.  The Avdonin formulation nevertheless exhibits cancellation invisible
to pointwise Kadec control.  For example,
$\delta_k=(-1)^k\alpha r$ with $1/4<\alpha<1/2$ has vanishing two-term
averages and remains separated, although it may lie outside the Kadec radius.
\end{remark}

\section{Generalized divided differences and confluent sampling}
\label{sec:gdd}

Unless a statement explicitly assumes only completeness, the frame and
operator hypotheses of \cref{sec:perturbations} are in force throughout
this section.

Classical divided-difference theory provides the appropriate replacement for
pairwise separation when temporal nodes occur in uniformly bounded clusters.
We use the formulation of Avdonin--Ivanov \cite{AI95,AI02,AI09}
and Avdonin--Moran \cite{AM01}; see also the weakened-gap Ingham--Beurling theorems of
Baiocchi--Komornik--Loreti \cite{BKL02} and their vector-valued extension
\cite{BKM09}.

\begin{definition}
\label{def:weak-gap-general}
Let $\Lambda=(\lambda_k)_{k\ge0}$ be a nondecreasing real sequence; repetitions
are allowed and are counted with multiplicity.  We say that $\Lambda$ satisfies
a \emph{weak $M$-gap condition} if there are $M\in\N$ and $\gamma>0$ such that
\begin{equation}\label{eq:weak-gap-general}
   \lambda_{k+M}-\lambda_k\ge M\gamma,
   \qquad k\ge0.
\end{equation}
Fix $0<\eta<\gamma$ and, without loss of generality, choose $\eta$ not
equal to any positive interpoint gap of the sequence.  A maximal consecutive
block
\[
   C_p=(\lambda_{p,1},\ldots,\lambda_{p,m_p})
\]
is called an $\eta$-close chain if every internal gap is less than $\eta$ and the gaps to
neighboring blocks, when present, are bigger than $\eta$.
\end{definition}

Condition \eqref{eq:weak-gap-general} implies $m_p\le M$: otherwise $M$
consecutive internal gaps would have sum $<M\eta<M\gamma$.  It also makes the
connection with the classical Avdonin--Moran clustering explicit.  Each residue
subsequence $(\lambda_{a+kM})_{k\ge0}$ is uniformly separated with gap at least
$M\gamma$.  Thus $\Lambda$ is a union of at most $M$ uniformly discrete
subsequences; a sequence with this property is called \emph{relatively
uniformly discrete}.  If the disc radius in the construction of \cite{AM01} is chosen
as $\eta/2$, then $\eta/2<\gamma/2$, which lies below the admissible radius
coming from this decomposition into $M$ separated subsequences, and the connected components are
exactly the maximal $\eta$-close chains defined above.  Conversely, for a real
sequence relative uniform discreteness is equivalent to a weak $M$-gap
condition for some $M$; compare \cite[Lemma~3]{AI02}.  In this subsection the
counting functions and densities count multiplicities.

\begin{definition}
\label{def:gdd-general}
Let $X$ be a complex Banach space and let $F:\C\to X$ be entire.  For real
nodes $\mu_1,\ldots,\mu_q$, not necessarily distinct, define
\begin{equation}\label{eq:hermite-genocchi-general}
 [\mu_1,\ldots,\mu_q]F
 =\int_{\Delta_{q-1}}
   F^{(q-1)}\!\left(\sum_{\nu=1}^q t_\nu\mu_\nu\right)\,d\sigma(t),
\end{equation}
where
\[
 \Delta_{q-1}=\{t_\nu\ge0:\ \sum_{\nu=1}^q t_\nu=1\}
\]
and $d\sigma$ denotes the affine simplex measure normalized so that
$\sigma(\Delta_{q-1})=1/(q-1)!$ (equivalently, the measure arising from the
standard nested-integral form of the Hermite--Genocchi formula).  This is the Hermite--Genocchi extension of the ordinary
divided difference.  If the nodes are distinct, then
\begin{equation}\label{eq:gdd-distinct-general}
 [\mu_1,\ldots,\mu_q]F
 =\sum_{\nu=1}^q
   \frac{F(\mu_\nu)}{\prod_{\ell\ne\nu}(\mu_\nu-\mu_\ell)},
\end{equation}
and if all nodes coincide at $\mu$, then
\begin{equation}\label{eq:gdd-confluent-general}
 [\mu,\ldots,\mu]F=\frac{F^{(q-1)}(\mu)}{(q-1)!}.
\end{equation}
For $F_t(\zeta)=e^{i\zeta t}$ we write
\[
   E[\mu_1,\ldots,\mu_q](t)
   =[\mu_1,\ldots,\mu_q]F_t.
\]
\end{definition}

The formula \eqref{eq:hermite-genocchi-general} is valid as a Bochner integral
and shows immediately that GDDs depend continuously on colliding nodes.  In
particular,
\begin{equation}\label{eq:exp-gdd-collision-general}
   E[\mu,\ldots,\mu](t)
   =\frac{(it)^{q-1}}{(q-1)!}e^{i\mu t}.
\end{equation}
For the operator-valued entire function $F_j(\zeta)=T^\zeta f_j$ it gives
\begin{equation}\label{eq:dynamical-gdd-collision-general}
 [\mu,\ldots,\mu]F_j
 =\frac{1}{(q-1)!}T^\mu(\Log T)^{q-1}f_j.
\end{equation}

For an $\eta$-close-chain decomposition $(C_p)$ of $\Lambda$, define the
\emph{dynamical GDD family}
\begin{equation}\label{eq:dynamical-gdd-family-general}
 \mathcal G(T,\bm f;\Lambda)
 =\left\{
 [\lambda_{p,1},\ldots,\lambda_{p,q}]
       (\zeta\mapsto T^\zeta f_j):
 p,\ 1\le q\le m_p,\ 1\le j\le m
 \right\}.
\end{equation}
Different enumerations inside clusters of bounded size and diameter
give families related by uniformly bounded, uniformly invertible
coefficient transformations.  Our convention uses the temporal order;
the dynamical vectors themselves need not be linearly independent.

\begin{remark}
\label{rem:gdd-channel-cyclic-general}
Let $J\subset\{1,\ldots,m\}$.  Every finite dynamical GDD formed from real
temporal nodes in the channels $j\in J$ belongs to $\Hh_J$.  Indeed, if
$x\perp\Hh_J$, then the proof of
\cref{prop:fractional-channel-cyclic-general} shows that
$F_{x,j}(z)=\langle T^zf_j,x\rangle$ vanishes identically for every $j\in J$.
All divided differences of $F_{x,j}$ therefore vanish, which proves the
claim by orthogonality.
\end{remark}

Recall that $L_{\rm mult}(\Lambda)$ is the logarithmic block density
\eqref{eq:log-block-density-general} with multiplicities counted.
The following theorem extends \cref{thm:sectorial-completeness-general}
to nested divided differences.

\begin{theorem}
\label{thm:gdd-completeness-general}
Let $T$ be invertible, assume
$\sigma(T)\subset\widehat\Sigma_c$, and suppose that the integer multiorbit
$\{T^nf_j\}_{n,j}$ is complete.  Let $\Lambda$ be a locally finite
nondecreasing temporal multisequence in $[0,\infty)$, partitioned into finite consecutive
clusters $C_p$, with all copies of a repeated node placed in the same cluster.
If
\begin{equation}\label{eq:gdd-log-completeness-general}
  L_{\rm mult}(\Lambda)> c/\pi,
\end{equation}
then the associated nested dynamical GDD family is complete.
\end{theorem}

\begin{proof}
Let $h$ be orthogonal to all of the GDD vectors and put
\[
  F_{h,j}(z)=\langle T^zf_j,h\rangle.
\]
The growth argument in the proof of
\cref{thm:sectorial-completeness-general} shows that each $F_{h,j}$ is entire
of exponential type and of imaginary-axis type at most $c$.  On a cluster
$C_p=(\lambda_{p,1},\ldots,\lambda_{p,m_p})$, orthogonality gives
\[
 [\lambda_{p,1}]F_{h,j}
 =[\lambda_{p,1},\lambda_{p,2}]F_{h,j}
 =\cdots
 =[\lambda_{p,1},\ldots,\lambda_{p,m_p}]F_{h,j}=0.
\]
Newton--Hermite interpolation then shows that the zero divisor of $F_{h,j}$
contains all nodes of $C_p$ with their prescribed multiplicities.  Therefore
its positive real zero divisor contains the positive part of $\Lambda$.
Rubel's theorem counts multiplicity, so
\eqref{eq:gdd-log-completeness-general} forces $F_{h,j}\equiv0$ for every
$j$.  Completeness of the integer multiorbit gives $h=0$.
\end{proof}

For each cluster and channel, write
\[
  \mathcal B_{p,j}
  =\left\{
  [\lambda_{p,1},\ldots,\lambda_{p,q}]
  (\zeta\mapsto T^\zeta f_j):1\le q\le m_p
  \right\}.
\]
We say that the GDD family has the \emph{cluster-block complement property}
if one side of every partition that keeps each block $\mathcal B_{p,j}$ intact
is complete.

\begin{corollary}
\label{cor:gdd-block-complement-general}
Under the hypotheses of \cref{thm:gdd-completeness-general}, assume
\(
  L_{\rm mult}(\Lambda)>{2c}/{\pi}.
\)
Then the GDD family has the cluster-block complement property if and only if
the channel-cyclic condition \eqref{eq:channel-complement-general} holds.  In
particular, it has the cluster-block complement property in the singly
generated case.
\end{corollary}

\begin{proof}
Necessity follows by partitioning all GDD blocks according to their channel
and applying \cref{rem:gdd-channel-cyclic-general}.  For sufficiency, suppose a block-respecting partition has two incomplete
sides, and choose nonzero vectors $x$ and $y$ orthogonal to them.  For each
channel, the product
\[
  G_j(z)=\langle T^zf_j,x\rangle\langle T^zf_j,y\rangle
\]
has every temporal cluster in its zero divisor with the prescribed
multiplicities: the full nested GDD block annihilates one of the two factors.
The product has imaginary-axis type at most $2c$, so Rubel's theorem gives
$G_j\equiv0$.  The no-zero-divisor argument from
\cref{thm:complement-property-general} then contradicts
\eqref{eq:channel-complement-general}.
\end{proof}

\begin{remark}
The block qualification is essential.  An arbitrary partition may separate
value and derivative data inside one cluster; the resulting conditions do not
force the full cluster multiplicity into the zero divisor of the product of
two coefficient functions.  Consequently, the corollary does not by itself
assert the ordinary complement property or real phase retrieval for a GDD
frame.
\end{remark}

\begin{lemma}
\label{lem:exponential-gdd-frame-general}
Let $\Omega=(\omega_k)_{k\in\Z}$ be a real relatively uniformly discrete
sequence, with multiplicities, and let its standard close-chain GDD family be
formed as in \cref{def:gdd-general}.  If $I$ is a bounded interval and
\begin{equation}\label{eq:gdd-density-classical-general}
   D^-(\Omega)>{|I|}/{2\pi},
\end{equation}
then the exponential GDD family is a frame for $L^2(I)$.
\end{lemma}

\begin{proof}
Write $T_I=|I|$.  Avdonin and Moran
\cite[Theorem~3(i)]{AM01} show that on $(0,T_I)$ one can select from the
standard GDD family a subfamily that is a Riesz basis whenever
$T_I<2\pi D^-(\Omega)$.  We first show that the \emph{full} standard GDD
family is a frame on $(0,T_I)$.

Its cluster cardinalities are uniformly bounded.  For a fixed order $q$, the
Hermite--Genocchi representation gives
\[
 E[\omega_{p,1},\ldots,\omega_{p,q}](t)
 =\int_{\Delta_{q-1}}
    (it)^{q-1}e^{i\xi_p(s)t}\,d\sigma(s),
 \qquad
 \xi_p(s)=\sum_{\nu=1}^q s_\nu\omega_{p,\nu}.
\]
Distinct close chains have uniformly separated convex hulls.  Hence, for every
fixed simplex parameter $s$, the real frequencies $(\xi_p(s))_p$ are uniformly
separated, with a separation constant independent of $s$.  The standard upper
Ingham estimate on $(0,T_I)$ therefore gives
\[
 \left\|\sum_p a_p e^{i\xi_p(s)t}\right\|_{L^2(0,T_I)}^2
 \le C_I\sum_p|a_p|^2
\]
uniformly in $s$.  Since $(0,T_I)$ is bounded, multiplication by
$(it)^{q-1}$ is bounded on $L^2(0,T_I)$.  Minkowski's integral inequality then
yields the same Bessel estimate for the order $q-1$ GDDs.  Summing over the
finitely many orders proves that the full GDD family is Bessel.  Since it
contains the Riesz basis selected in \cite[Theorem~3(i)]{AM01}, it is a frame
for $L^2(0,T_I)$.

It remains to transfer this full-frame conclusion to an arbitrary interval of
the same length.  Let $I=(a,a+T_I)$ and let
\[
  (U_af)(t)=f(t+a),\qquad 0<t<T_I,
\]
be the translation unitary from $L^2(I)$ onto $L^2(0,T_I)$.  Inside each
cluster, the divided-difference product rule gives the explicit triangular
identity
\begin{align*}
 U_aE[\omega_{p,1},\ldots,\omega_{p,q}](t)
 &= [\omega_{p,1},\ldots,\omega_{p,q}]
      \bigl(\zeta\mapsto e^{i\zeta t}e^{ia\zeta}\bigr)\\
 &=\sum_{k=1}^q
    E[\omega_{p,1},\ldots,\omega_{p,k}](t)
    [\omega_{p,k},\ldots,\omega_{p,q}]
       (\zeta\mapsto e^{ia\zeta}).
\end{align*}
Thus translation is a triangular linear combination of the first $q$
ordered GDDs in the same cluster on $(0,T_I)$, with unimodular diagonal
coefficient $e^{ia\omega_{p,q}}$.  More explicitly, the Hermite--Genocchi
formula gives, for every suffix of length $s$,
\[
 \left|[\mu_1,\ldots,\mu_s]
       (\zeta\mapsto e^{ia\zeta})\right|
 \le \frac{|a|^{s-1}}{(s-1)!},
\]
including at repeated nodes.  Since the cluster cardinalities are uniformly
bounded, the block matrices are uniformly bounded independently of the
absolute locations of the clusters.  Translation by $-a$ produces their
exact inverses and satisfies the same type of estimate.  Thus $U_a$ carries
the full standard GDD family
on $I$ to a uniformly bounded and uniformly invertible blockwise transform of
the full standard GDD frame on $(0,T_I)$.  Such a blockwise transform preserves
the frame property, proving the assertion on $I$.
\end{proof}

We next show that the compact exponential reduction from
\cref{lem:compact-localization-general} is stable under arbitrary collisions of
bounded multiplicity.

\begin{lemma}
\label{lem:gdd-compact-localization-general}
Assume that $\Lambda=(\lambda_k)_{k\ge0}\subset[0,\infty)$ satisfies a weak $M$-gap condition
and let $(C_p)$ be an $\eta$-close-chain decomposition with
$0<\eta<\min\{1,\gamma\}$.  For
$C_p=(\lambda_{p,1},\ldots,\lambda_{p,m_p})$, set
\[
   n_p=\lfloor\lambda_{p,1}\rfloor,
   \qquad
   \alpha_{p,q}=\lambda_{p,q}-n_p.
\]
Then the offsets $\alpha_{p,q}$ lie in one fixed compact interval and the map
$p\mapsto n_p$ has uniformly bounded multiplicity.

For $1\le q\le m_p$ and $1\le j\le m$, let
\begin{align}
 x_{p,q,j}
 &= [\lambda_{p,1},\ldots,\lambda_{p,q}]
      (\zeta\mapsto A^\zeta g_j),\label{eq:gdd-x-general}\\
 w_{p,q,j}
 &= [\alpha_{p,1},\ldots,\alpha_{p,q}]
    \bigl(\alpha\mapsto PT_{q_\alpha}(z^{n_p}e_j)\bigr),
    \label{eq:gdd-w-general}
\end{align}
where $q_\alpha$ is given by \eqref{eq:q-alpha-general}.
The synthesis operators of $(x_{p,q,j})$ and $(w_{p,q,j})$ differ by a
compact operator.  Moreover,
\begin{equation}\label{eq:gdd-localized-vector-general}
   w_{p,q,j}
   =PP_+\!\left(
      \chi(e^{it})
      E[\lambda_{p,1},\ldots,\lambda_{p,q}](t)e_j
   \right).
\end{equation}
\end{lemma}

\begin{proof}
Every close chain has at most $M$ elements and diameter less than
$(M-1)\eta$.  Hence
\[
   0\le\alpha_{p,q}<R_0:=1+(M-1)\eta.
\]
The first points of distinct maximal chains are separated by at least $\eta$;
therefore only a bounded number of them can have the same integer part $n_p$.

Recall the compact-valued entire function
\[
   C_\alpha=A^\alpha P-PT_{q_\alpha}
\]
from \cref{lem:compact-localization-general}.  Translation invariance of divided
differences and $A^{n_p}g_j=Pz^{n_p}e_j$ give
\begin{equation}\label{eq:gdd-error-column-general}
 x_{p,q,j}-w_{p,q,j}
 = [\alpha_{p,1},\ldots,\alpha_{p,q}]C_\alpha
      (z^{n_p}e_j).
\end{equation}
Expand $C_\alpha$ at the origin,
\[
   C_\alpha=\sum_{\ell=0}^\infty C_\ell\alpha^\ell,
\]
where every $C_\ell$ is compact.  For $\ell\ge q-1$, the divided difference of
the monomial is the complete homogeneous symmetric polynomial
\[
 [\alpha_1,\ldots,\alpha_q]\alpha^\ell
 =h_{\ell-q+1}(\alpha_1,\ldots,\alpha_q),
\]
and it is zero for $\ell<q-1$.  If $0\le\alpha_\nu\le R_0$, then
\begin{equation}\label{eq:h-bound-general}
 |h_s(\alpha_1,\ldots,\alpha_q)|
 \le \binom{s+q-1}{q-1}R_0^s.
\end{equation}
Choose $R_1>R_0$.  The Banach-valued Cauchy estimate on $|\alpha|=R_1$ gives
$\|C_\ell\|\le M_{R_1}R_1^{-\ell}$.  Consequently, for each
$q\le M$,
\begin{equation}\label{eq:gdd-norm-sum-general}
 \sum_{\ell=q-1}^\infty
 \|C_\ell\|
 \sup_p
 |h_{\ell-q+1}(\alpha_{p,1},\ldots,\alpha_{p,q})|<\infty.
\end{equation}

For fixed $q$, let $J_q$ map the coordinate indexed by $(p,j)$ to
$z^{n_p}e_j$.  The bounded multiplicity of $p\mapsto n_p$ makes $J_q$ bounded.
Let $M_{\ell,q}$ be the diagonal multiplier with entries
$h_{\ell-q+1}(\alpha_{p,1},\ldots,\alpha_{p,q})$.  By
\eqref{eq:gdd-error-column-general}, the synthesis error for the order $q-1$
GDDs is
\[
   \sum_{\ell=q-1}^\infty C_\ell J_qM_{\ell,q}.
\]
Every summand is compact, and \eqref{eq:gdd-norm-sum-general} gives convergence
in operator norm.  Thus the error is compact.  Summing over
$q=1,\ldots,M$ preserves compactness.

Finally, divided differences commute with the bounded linear Toeplitz map and
\[
 q_\alpha(e^{it})z^{n_p}
 =\chi(e^{it})e^{i(n_p+\alpha)t}.
\]
Taking the divided difference in $\alpha$ proves
\eqref{eq:gdd-localized-vector-general}.
\end{proof}

\begin{corollary}
\label{cor:gdd-fredholm-transfer-general}
In the setting of \cref{lem:gdd-compact-localization-general}, let
$\mathcal Y_d$ and $V_d=XR_d^*$ be as in
\cref{thm:fredholm-transfer-general}.  Define the projected exponential GDD
family
\[
  \mathcal E_d^{\rm GDD}(\Lambda)
  =\left\{
    P_{\mathcal Y_d}\bigl(
      E[\lambda_{p,1},\ldots,\lambda_{p,q}](t)e_j
    \bigr):p,\ 1\le q\le m_p,\ 1\le j\le m
  \right\}.
\]
If $D_{T,\Lambda}^{\rm GDD}$ and
$D_{\mathcal E_d^{\rm GDD}(\Lambda)}$ denote the synthesis operators of the
dynamical and projected exponential GDD families, respectively, then
\begin{equation}\label{eq:gdd-fredholm-transfer-general}
  D_{T,\Lambda}^{\rm GDD}
  =V_dD_{\mathcal E_d^{\rm GDD}(\Lambda)}+K_{\Lambda}^{\rm GDD}
\end{equation}
for a compact operator $K_{\Lambda}^{\rm GDD}$.  Hence upper and lower
semi-Fredholm properties pass between the two systems, and their Fredholm
indices agree whenever either synthesis operator is Fredholm.
\end{corollary}

\begin{proof}
The upper estimate in the proof of
\cref{lem:exponential-gdd-frame-general} uses only the uniform cluster
bound and separation of their convex hulls.  Thus the projected
exponential GDD family is Bessel, without a density assumption.
By \eqref{eq:gdd-localized-vector-general}, the localized GDD column
$w_{p,q,j}$ equals
\[
  R_d^*P_{\mathcal Y_d}
  \bigl(E[\lambda_{p,1},\ldots,\lambda_{p,q}](t)e_j\bigr).
\]
The synthesis difference between these columns and the model dynamical GDD
columns is compact by \cref{lem:gdd-compact-localization-general}.  Multiplying
by the similarity $X$ gives \eqref{eq:gdd-fredholm-transfer-general}; the
Fredholm assertions follow exactly as in
\cref{thm:fredholm-transfer-general}.
\end{proof}

The compact comparison and the multiplicity-counted completeness
criterion give the following frame theorem for clustered sampling.

\begin{theorem}
\label{thm:gdd-sectorial-general}
Assume that \eqref{eq:intro-frame} is a frame, $T$ is invertible, and
\eqref{eq:intro-sector} and \eqref{eq:essential-sector-general} hold.
Let $\Lambda=(\lambda_k)_{k\ge0}\subset[0,\infty)$ be a nondecreasing
sequence satisfying a weak $M$-gap condition, and count multiplicities
in its densities.  If
\begin{equation}\label{eq:gdd-essential-density-general}
  D_+^-(\Lambda)>c_{\mathrm e}/\pi,
\end{equation}
then the dynamical GDD family \eqref{eq:dynamical-gdd-family-general}
is Bessel and its synthesis operator has closed range of finite
codimension.  If also
\begin{equation}\label{eq:gdd-log-frame-density-general}
  L_{\rm mult}(\Lambda)>c/\pi,
\end{equation}
then it is a frame.  In particular, the single condition
\begin{equation}\label{eq:gdd-sectorial-density-general}
  D_+^-(\Lambda)>c/\pi
\end{equation}
suffices.
\end{theorem}

\begin{proof}
Fix an admissible $\eta$-close-chain decomposition with
$0<\eta<\min\{1,\gamma\}$ and adjoin the negative integers as singleton
clusters:
\[
   \widetilde\Lambda=\Lambda\cup\{-1,-2,-3,\ldots\}.
\]
The weak-gap condition makes $\widetilde\Lambda$ relatively uniformly
discrete.  The same threshold $\eta$ is admissible in the classical
close-chain construction: put $b=\max\{1,1/\gamma\}$.  Every interval
of length $H$ contains at most $bH+M+1$ augmented nodes.  Hence,
for a nondecreasing two-sided enumeration $(\mu_k)$ of the augmented
sequence,
\[
  \mu_{k+N}-\mu_k\ge (N-M)/b.
\]
Since $b\eta<1$, choose $N$ large enough that
$(N-M)/(bN)>\eta$.  Its $N$ residue classes have separation at least
$(N-M)/b$.  Put $\delta=(N-M)/b$; then disc radius $\eta/2$ is below
the radius $\delta/(2N)$ in \cite[Lemma~1]{AM01}.  Negative integers remain
singleton chains because $\eta<1$.
The multiplicity-counted form of
\cref{lem:augmented-density-general} gives
\[
   D^-(\widetilde\Lambda)=\min\{1,D_+^-(\Lambda)\}.
\]
Choose
\[
   c_{\mathrm e}<d<\pi\min\{1,D_+^-(\Lambda)\}.
\]
By \cref{lem:exponential-gdd-frame-general}, the vector-valued exponential GDD
family associated with $\widetilde\Lambda$ is a frame for
$L^2(-d,d;\mathbb C^m)$, and its orthogonal projection is therefore a frame
for $\mathcal Y_d$.

Split the projected family into its positive-cluster part and the auxiliary
negative singleton exponentials.  The latter have compact analysis operator
exactly as in the proof of
\cref{thm:separated-general}.  Hence, the synthesis operator of the positive
projected exponential GDD family is lower semi-Fredholm.  The confluent
transfer identity \eqref{eq:gdd-fredholm-transfer-general} now makes the
dynamical GDD synthesis operator lower semi-Fredholm as well.

This proves the closed-range assertion.  Under
\eqref{eq:gdd-log-frame-density-general},
\cref{thm:gdd-completeness-general} gives completeness.  The synthesis
range is then both dense and closed, hence equal to $\mathcal H$.
Finally, \eqref{eq:gdd-sectorial-density-general} implies both density
conditions by $c_{\mathrm e}\le c$ and the multiplicity-counted form
of \cref{lem:uniform-log-general}.
\end{proof}

\begin{remark}
\label{rem:gdd-versus-raw-general}
The divided differences in \cref{thm:gdd-sectorial-general} preserve
confluent directions as temporal nodes coalesce.  For two distinct nodes,
\[
 [\lambda,\mu](\zeta\mapsto T^\zeta f_j)
 =\frac{T^\lambda-T^\mu}{\lambda-\mu}f_j
 \longrightarrow T^\lambda(\Log T)f_j
 \quad(\mu\to\lambda).
\]
Thus GDDs retain the missing confluent direction when raw samples coalesce.  The
pairing obstruction in \cref{prop:pairing-obstruction-general} shows that no
uniform theorem can recover this direction from the unscaled raw samples alone.
For clusters with distinct nodes, the GDDs and raw orbit vectors span the same
finite-dimensional cluster subspace; at an exact collision the natural limiting
object is instead the confluent subspace generated by the logarithmic
derivatives in \eqref{eq:dynamical-gdd-collision-general}.
\end{remark}

\begin{corollary}
\label{cor:gdd-cell-jitter-general}
Let $\lambda_k=rk+j_k$ with $0\le j_k<r$.  If
\begin{equation}\label{eq:gdd-cell-full-range-general}
   rc<\pi,
\end{equation}
then the dynamical GDD family obtained from the close pairs of the sequence
$(\lambda_k)$ is a frame.  Every close chain has size at most two, so the only
nontrivial replacements are
\begin{equation}\label{eq:gdd-pair-replacement-general}
   T^\lambda f_j,
   \qquad
   \frac{T^\lambda-T^\mu}{\lambda-\mu}f_j,
\end{equation}
with the second vector interpreted as $T^\lambda(\Log T)f_j$ in the
confluent limiting case (or when coincident endpoints are admitted).
\end{corollary}

\begin{proof}
The sequence is increasing and
\(
   \lambda_{k+2}-\lambda_k>r.
\)
Hence, it satisfies a weak $2$-gap condition, for example with any
$\gamma<r/2$.  Its lower uniform density, counting the one point in every
$r$-cell, equals $1/r$.  Condition \eqref{eq:gdd-cell-full-range-general} is
exactly $1/r>c/\pi$.  Apply \cref{thm:gdd-sectorial-general}.  If the clustering
threshold is chosen below $r/2$, no close chain can contain three points.
\end{proof}

The pair-collision mechanism underlying this corollary is illustrated in
\cref{fig:gdd-collision-general}; higher-order collisions are treated by the
same Hermite--Genocchi construction.

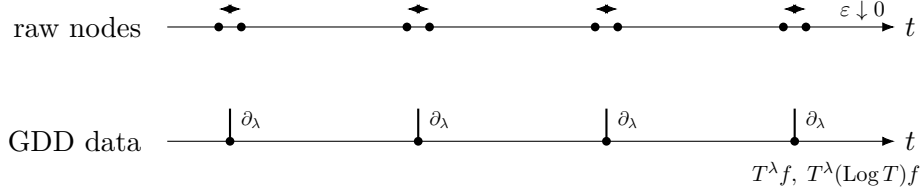
\begin{figure}[H]
\centering
\begin{tikzpicture}[x=.83cm,y=.72cm,>=Latex]
  \draw[->] (0,2.65)--(11.6,2.65) node[right] {$t$};
  \draw[->] (0,.55)--(11.6,.55) node[right] {$t$};
  \node[anchor=east] at (-.25,2.65) {raw nodes};
  \node[anchor=east] at (-.25,.55) {GDD data};
  \foreach \x in {1.0,4.0,7.0,10.0}{
    \fill (\x-.18,2.65) circle (1.65pt);
    \fill (\x+.18,2.65) circle (1.65pt);
    \draw[<->,thin] (\x-.18,2.98)--(\x+.18,2.98);
    \fill (\x,.55) circle (1.65pt);
    \draw[thick] (\x,.55)--(\x,1.15);
    \node[scale=.67,anchor=west] at (\x+.07,.95) {$\partial_\lambda$};
  }
  \node[scale=.76,anchor=west] at (10.6,2.95) {$\varepsilon\downarrow0$};
  \node[scale=.76,anchor=west] at (9.2,-.05)
    {$T^\lambda f,\;T^\lambda(\Log T)f$};
\end{tikzpicture}
\caption{GDD preconditioning through a pair collision.  Two nearly coincident
raw temporal samples are replaced by a value and a divided difference; at exact
collision the second datum becomes a logarithmic derivative.  Higher-order
clusters lead analogously to higher powers of $\Log T$.}
\label{fig:gdd-collision-general}
\end{figure}

\subsection{Raw one-point-per-cell sampling}
For raw one-point-per-cell samples, the closed-range condition depends
on the essential spectral bound $c_{\mathrm e}$, while completeness is
controlled by the full spectral bound $c$.
\begin{theorem}
\label{thm:cell-jitter-general}
Under the standing frame hypotheses, let $r>0$ satisfy
\begin{equation}\label{eq:cell-jitter-threshold-general}
   rc<\pi,\qquad 2rc_{\mathrm e}<\pi.
\end{equation}
For every choice
\begin{equation}\label{eq:cell-jitter-times-general}
   0\le j_k<r,\qquad k\in\Nzero,
\end{equation}
the raw sampled family
\begin{equation}\label{eq:cell-jitter-family-general}
   \{T^{rk+j_k}f_j:k\in\Nzero,\ 1\le j\le m\}
\end{equation}
is a frame.  In particular, $rc<\pi/2$ is sufficient without any
restriction on the essential spectrum beyond $c_{\mathrm e}\le c$.
\end{theorem}

\begin{proof}
Write $\lambda_k=rk+j_k$ and $\Lambda=\{\lambda_k:k\ge0\}$.
There is one point in each half-open $r$-cell, so
\[
 D_+^-(\Lambda)=1/r,\qquad L(\Lambda)\ge1/r.
\]
Moreover, $\lambda_{k+2}-\lambda_k>r$.  Thus the even subsequence
$\Lambda_0=\{\lambda_{2k}:k\ge0\}$ is uniformly separated and
has $D_+^-(\Lambda_0)=1/(2r)>c_{\mathrm e}/\pi$.
By the closed-range assertion in \cref{thm:separated-general}, its
dynamical synthesis operator $D_0$ has closed range of finite codimension.

The full temporal sequence has bounded integer-part multiplicity,
so its synthesis operator $D$ is bounded by
\cref{thm:fredholm-transfer-general}.  Since
$\Ran D_0\subset\Ran D$, the latter range is closed and of finite
codimension as well: its image in the finite-dimensional quotient
$\mathcal H/\Ran D_0$ is a closed subspace.
Finally, $L(\Lambda)\ge1/r>c/\pi$ and
\cref{thm:sectorial-completeness-general} give completeness.
Hence $D$ is onto.
\end{proof}

\begin{remark}
\label{rem:ffc-comparison-general}
If $\sigma_{\mathrm e}(T)\subseteq\{1\}$, take $c_{\mathrm e}=0$.
Then \cref{thm:cell-jitter-general} gives raw one-point-per-cell
sampling throughout $rc<\pi$, without normality, diagonalizability,
or a Stolz-domain hypothesis.  This includes the one-point-per-cell
conclusion of \cite[Theorem~1.3]{KMMFFC26}: the separated spectrum
there lies in a Stolz domain accumulating only at $1$, and its
nonzero spectral points are bounded away from zero.

For a general essential spectrum, the second restriction
$2rc_{\mathrm e}<\pi$ has a different role from $rc<\pi$.
The former supplies closed range through a separated subsequence;
the latter supplies completeness of the full family.
The pairing obstruction in \cref{prop:pairing-obstruction-general}
occurs when $c_{\mathrm e}=c=\pi/2$, and shows that the universal raw
range remains $rc<\pi/2$.  GDD preconditioning recovers $rc<\pi$
without requiring a smaller essential spectral arc.
\end{remark}

\section{Sharpness and exact collisions}
\label{sec:sharpness}

The standing frame hypotheses remain in force.  The following examples
establish sharpness of the sampling thresholds and illustrate exact
collisions in the GDD theorem.

\begin{example}
\label{ex:density-sharpness-general}
Let $0<c<\pi$, $0<\rho<1$, and
\[
   T=\rho\begin{pmatrix}e^{ic}&0\\0&e^{-ic}\end{pmatrix},
   \qquad f=\begin{pmatrix}1\\1\end{pmatrix}
\]
on $\C^2$.  The orbit $(T^nf)_{n\ge0}$ is a frame: it is Bessel because of
the factor $\rho^n$, and $f,Tf$ span $\C^2$.  At the endpoint
$r=\pi/c$,
\[
   T^{rk}f
   =(-\rho^r)^k\begin{pmatrix}1\\1\end{pmatrix},
\]
so the progression spans only one dimension.  Hence the strict inequalities
$D_+^-(\Lambda)>c/\pi$ and $rc<\pi$ cannot be weakened uniformly.

The same example also tests the complement-property constant.  Set
$r=\pi/(2c)$.  Then
\[
 T^{rk}f=\rho^{rk}\begin{pmatrix}i^k\\(-i)^k\end{pmatrix}.
\]
The even subfamily spans $\mathbb C(1,1)^T$ and the odd subfamily
spans $\mathbb C(1,-1)^T$, so the complement property fails at
$L(r\mathbb N_0)=2c/\pi$.  For a real example, take $T$ to be
$\rho$ times the planar rotation by $c$ and $f=(1,0)^T$.
At the same step $r$, the two parity subfamilies span the two real
coordinate axes.  Thus the doubled threshold is sharp also for
the real phase-retrieval conclusion.
\end{example}

\subsection{Pairing obstruction for raw cell jitter}

The following singular-inner-function model shows that the universal
range $rc<\pi/2$ in \cref{thm:cell-jitter-general} is sharp for raw
one-point-per-cell sampling.

\begin{proposition}
\label{prop:pairing-obstruction-general}
There are an invertible operator $T$, a vector $f$ whose integer orbit is a frame,
and a sequence
\[
   \lambda_k=k+j_k,
   \qquad 0\le j_k<1,
\]
such that
\[
   \sigma(T)=\{i,-i\}\subset\{e^{it}:|t|\le\pi/2\}
\]
and $(T^{\lambda_k}f)_{k\ge0}$ is not a frame.  Thus the universal range
$rc<\pi/2$ in \cref{thm:cell-jitter-general} cannot be replaced uniformly by
$rc<\pi$.
\end{proposition}

\begin{proof}
Fix $a>0$ and let
\[
   \vartheta(z)=\exp\!\left(-a\frac{1+z}{1-z}\right),
   \qquad
   \theta(z)=\vartheta(-z^2)
   =\exp\!\left(-a\frac{1-z^2}{1+z^2}\right).
\]
The singular inner function $\theta$ has equal atoms at $i$ and $-i$.  Put $K_\theta=H^2\ominus\theta H^2$ and
$S_\theta=P_{K_\theta}M_z|_{K_\theta}$.  Since $\theta$ has no zeros in
$\mathbb D$ and its boundary spectrum is $\{i,-i\}$, the standard spectral
formula for model operators \cite{Nikolski02,Sarason94} gives
$\sigma(S_\theta)=\{i,-i\}$.  In particular, $S_\theta$ is invertible.
Take $T=S_\theta$ and $f=P_{K_\theta}1$.  Its integer orbit is the canonical
Parseval frame.
Define isometries
\[
   W_0h(z)=h(-z^2),
   \qquad W_1h(z)=zh(-z^2).
\]
Then
\begin{equation}\label{eq:pairing-decomposition-general}
   K_\theta=W_0K_\vartheta\oplus W_1K_\vartheta,
\end{equation}
and every odd canonical vector $T^{2q+1}f$ lies in the second summand.

By norm continuity of $s\mapsto T^sf$, choose
$0<\varepsilon_q<\min\{1/2,2^{-q}\}$ so that
\[
   \norm{T^{2q+1-\varepsilon_q}f-T^{2q+1}f}
   +
   \norm{T^{2q+1+\varepsilon_q}f-T^{2q+1}f}
   \le2^{-q-1}.
\]
Set
\[
   \lambda_{2q}=2q+1-\varepsilon_q,
   \qquad \lambda_{2q+1}=2q+1+\varepsilon_q.
\]
The resulting temporal pattern is depicted in
\cref{fig:pairing-general}.  There is one point in every unit cell.  The synthesis operator of
$(T^{\lambda_k}f)$ differs by a Hilbert--Schmidt operator from the synthesis
operator of the duplicated family
\[
   y_{2q}=y_{2q+1}=T^{2q+1}f.
\]
The latter has range in $W_1K_\vartheta$.  If the perturbed synthesis operator
were onto $K_\theta$, projection onto the infinite-dimensional summand
$W_0K_\vartheta$ would give a compact surjection, which is impossible.
\end{proof}

The same example shows why the separation hypothesis in Avdonin's
theorem is essential.

\begin{remark}
\label{rem:avdonin-pairing-general}
Relative to the progression $k$, the errors in the construction are
\[
   \delta_{2q}=1-\varepsilon_q,
   \qquad
   \delta_{2q+1}=\varepsilon_q.
\]
After subtracting the common offset $1/2$, the aligned two-term averages
are exactly zero, while the shifted two-term averages have absolute value
\[
   \frac12|\varepsilon_q-\varepsilon_{q+1}|<\frac14.
\]
Thus the strict $N=2$ Avdonin block-average condition holds.  Nevertheless,
\[
   \lambda_{2q+1}-\lambda_{2q}=2\varepsilon_q\longrightarrow0,
\]
so the frequencies are not uniformly separated.  Hence block-average control
cannot improve the arbitrary-cell threshold without a separation hypothesis.
\end{remark}

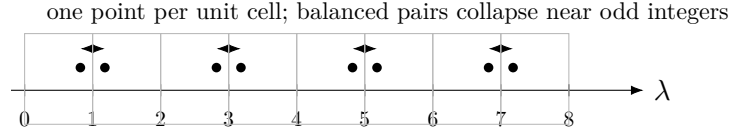
\begin{figure}[t]
\centering
\begin{tikzpicture}[x=.9cm,y=1cm,>=Latex]
  \draw[->] (-.2,0)--(9.1,0) node[right] {$\lambda$};
  \foreach \k in {0,...,8}{
    \draw (\k,0.08)--(\k,-0.08) node[below=3pt,scale=.75] {$\k$};
  }
  \foreach \q in {0,...,3}{
    \fill ({2*\q+1-.18},.30) circle (1.8pt);
    \fill ({2*\q+1+.18},.30) circle (1.8pt);
    \draw[<->,thin] ({2*\q+1-.18},.55)--({2*\q+1+.18},.55);
  }
  \foreach \k in {0,...,7}{
    \draw[gray!55] (\k,-.45) rectangle ({\k+1},.75);
  }
  \node[anchor=west,scale=.83] at (.2,1.02)
    {one point per unit cell; balanced pairs collapse near odd integers};
\end{tikzpicture}
\caption{The endpoint pairing pattern.  Each neighboring pair is balanced about an
odd integer but collapses to two copies of one canonical vector.  The system is a compact perturbation of a family contained in one
model-space summand.}
\label{fig:pairing-general}
\end{figure}

\subsection{Exact multiple collisions}

\Cref{thm:gdd-sectorial-general} yields functional-calculus frames
containing logarithmic derivatives of any fixed order.

\begin{example}
\label{ex:ffc-multiple-collision-general}
Fix $M\in\mathbb N$, $r>0$, and $s_0\ge0$, and assume
\begin{equation}\label{eq:ffc-multiple-density-general}
   rc<M\pi.
\end{equation}
At every temporal center $s_0+rk$, place a node of multiplicity $M$.  For
$1\le q\le M$, define
\begin{equation}\label{eq:ffc-confluent-functions-general}
   \varphi_{k,q}(z)
   =\frac{1}{(q-1)!}z^{s_0+rk}(\Log z)^{q-1}
\end{equation}
on the principal logarithm domain.  Then
\begin{equation}\label{eq:ffc-confluent-frame-general}
   \{\varphi_{k,q}(T)f_j:
      k\in\mathbb N_0,\ 1\le q\le M,\ 1\le j\le m\}
\end{equation}
is a frame for $\mathcal H$.
\end{example}

\begin{proof}
Let $\Lambda$ be the temporal multiset in which every $s_0+rk$ occurs
$M$ times.  It satisfies a weak $M$-gap condition and has
multiplicity-counted density $M/r$.  Condition
\eqref{eq:ffc-multiple-density-general} is exactly $M/r>c/\pi$, so
\cref{thm:gdd-sectorial-general} applies.  At a repeated node,
\eqref{eq:gdd-confluent-general} gives
\[
   [s_0+rk,\ldots,s_0+rk](\zeta\mapsto T^\zeta f_j)
   =\frac{1}{(q-1)!}T^{s_0+rk}(\Log T)^{q-1}f_j
   =\varphi_{k,q}(T)f_j,
\]
which proves \eqref{eq:ffc-confluent-frame-general}.
\end{proof}

\begin{remark}
For $M=2$ the confluent block consists of
$T^{s_0+rk}f_j$ and $T^{s_0+rk}(\Log T)f_j$.  In general one obtains
\[
   T^{s_0+rk}f_j,\quad T^{s_0+rk}(\Log T)f_j,\quad\ldots,\quad
   \frac{1}{(M-1)!}T^{s_0+rk}(\Log T)^{M-1}f_j.
\]
Thus exact temporal collisions yield Hermite-type functional-calculus
data.
\end{remark}

\section{Conclusion}

Compact Fredholm localization separates two sources of sampling
constraints.  The essential spectral arc controls the closed-range
estimate, whereas the full spectral arc controls completeness.
This distinction yields frames under
$D_+^-(\Lambda)>c_{\mathrm e}/\pi$ and $L(\Lambda)>c/\pi$,
with an analogous statement for multiplicity-counted GDD data.
The completeness and complement-property arguments require only
an invertible operator in a proper angular sector.

For raw one-point-per-cell sampling, the conditions are
$rc<\pi$ and $2rc_{\mathrm e}<\pi$.  They recover the full range
$rc<\pi$ when the essential spectrum is contained in $\{1\}$.
Uniformly over all admissible essential spectra, the range
$rc<\pi/2$ is sharp.
GDD data yield the full universal range $rc<\pi$ by retaining
confluent directions, including logarithmic derivatives at exact
collisions.  At doubled logarithmic density, the ordinary complement
property for raw samples reduces to the channel-cyclic condition;
the GDD analogue concerns partitions that keep each cluster block intact.

\section*{Acknowledgments}  The first author was supported in part by the Fulbright Global Scholar Award; he thanks the second author for his wonderful hospitality in Qu\'ebec. The second author acknowledges support from the Canada Research Chairs Program CRC-2022-00097 and
the NSERC Discovery Grant RGPIN-2024-04232.

\section*{Declaration of generative AI and AI-assisted technologies in the
manuscript preparation process}
During the preparation of this work, the authors used OpenAI's ChatGPT to
assist with organization, preliminary \LaTeX{} drafting and figures, algebraic consistency
checks, and bibliography formatting.  The authors reviewed and
edited all AI-assisted output, verified the mathematical arguments and
references, and take full responsibility for the content of the article.

\end{document}